\documentclass[final]{siamltex}
\usepackage{amssymb}
\usepackage{graphicx}
\newtheorem{Remark}[theorem]{Remark}

\title{Equivalence between minimal time and minimal norm control problems for the heat equation
}

\author{Shulin Qin\thanks{School of Mathematics and Statistics, Wuhan University, Wuhan, 430072, China (shulinqin@yeah.net).} \and Gengsheng Wang\thanks{School of Mathematics and Statistics, Computational Science Hubei Key Laboratory, Wuhan University, Wuhan, 430072, China (wanggs62@yeah.net).  The author was partially supported by the National Natural Science Foundation of China  under grant 11571264. }
}

\begin{document}

\maketitle

\begin{abstract}
This paper presents the equivalence between
minimal time and minimal norm control problems for internally controlled heat equations. The target  is an arbitrarily fixed bounded, closed and convex set with a nonempty interior in the state space. This study differs from [G. Wang and E. Zuazua, \textit{On the equivalence of minimal time and minimal norm controls for internally controlled heat equations}, SIAM J. Control Optim., 50 (2012), pp. 2938-2958] where the target set is the origin in the state space. When the target set is the origin
or a ball, centered at the origin, the minimal norm and the minimal time functions are continuous and strictly decreasing, and they are inverses of each other. However, when the target is located in other place of the state space,  the minimal norm function may be no longer monotonous and the range of the minimal time function may not be connected. These cause the main difficulty in our study. We overcome this difficulty by borrowing some idea from the classical raising sun lemma  (see, for instance,  Lemma 3.5 and Figure 5 on Pages 121-122 in [E. M. Stein and R. Shakarchi, \textit{Real Analysis: Measure Theory, Integration, and Hilbert Spaces}, Princeton University Press, 2005]).
\end{abstract}



\pagestyle{myheadings}
\thispagestyle{plain}
\markboth{SHULIN QIN AND GENGSHENG WANG}{EQUIVALENCE OF MINIMAL TIME AND NORM CONTROLS}

\section{Introduction}

Let $\Omega\subset \mathbb R^n$ $(n\in\mathbb N^+)$ be a bounded open domain with a $C^2$ boundary $\partial\Omega$. Let $\omega\subset\Omega$ be an open and nonempty subset with the characteristic function  $\chi_\omega$. Write $\mathbb R^+ \triangleq (0, +\infty)$. Consider the following two controlled heat equations:
\begin{eqnarray}\label{heat-eq-infty}
\left\{\begin{array}{lll}
        \partial_t y-\Delta y=\chi_\omega u  & \mbox{in} &\mathbb R^+\times\Omega,\\
        y=0& \mbox{on} &\mathbb R^+\times\partial\Omega,\\
        y(0)=y_0& \mbox{in} &\Omega
       \end{array}
\right.
\end{eqnarray}
and
\begin{eqnarray}\label{heat-eq-finite}
\left\{\begin{array}{lll}
        \partial_t y-\Delta y=\chi_\omega v  & \mbox{in} &\mathbb (0,T)\times\Omega,\\
        y=0& \mbox{on} &(0,T)\times\partial\Omega,\\
        y(0)=y_0& \mbox{in} &\Omega.
       \end{array}
\right.
\end{eqnarray}
Here, $T>0$, $y_0\in L^2(\Omega)$ and the controls $u$ and $v$ are  taken  from the spaces $L^\infty(\mathbb R^+; L^2(\Omega))$ and $L^\infty(0, T; L^2(\Omega))$, respectively. Denote by $y(\cdot; y_0, u)$ and $\hat y(\cdot; y_0, v)$ the solutions of (\ref{heat-eq-infty}) and (\ref{heat-eq-finite}), respectively. Throughout this paper,  $\|\cdot\|$ and $\langle\cdot,\cdot\rangle$ denote the usual norm and inner product of $L^2(\Omega)$, respectively.

Write $\mathcal{F}$ for the set consisting of all bounded closed convex subsets which have nonempty interiors in $L^2(\Omega)$.
For each $M \geq 0$, $y_0\in L^2(\Omega)$ and $Q \in \mathcal{F}$, we define the following minimal time control problem:
\begin{eqnarray}\label{TP}
 (TP)^{M,y_0}_{Q}: \;\; T(M,y_0,Q) \triangleq \inf\big\{\hat t\geq 0~:~\exists\,u\in\mathcal U^M \mbox{ s.t. } y(\hat t;y_0,u)\in Q \big\},
\end{eqnarray}
where
\begin{eqnarray*}
 \mathcal U^M\triangleq \big\{ u\in L^\infty(\mathbb R^+; L^2(\Omega))~:~ \|u(t)\|\leq M \mbox{ a.e. } t\in\mathbb R^+  \big\}.
\end{eqnarray*}
For  each $T > 0$, $y_0 \in L^2(\Omega)$ and $Q \in \mathcal{F}$, we define the following minimal norm control problem:
\begin{eqnarray}\label{I-1}
 (NP)^{T,y_0}_{Q}: \;\; N(T,y_0,Q)\triangleq \inf \big\{  \|v\|_{L^\infty(0,T;L^2(\Omega))}
 ~:~\hat y(T;y_0,v)\in Q  \big\}.
\end{eqnarray}
In these two problems, $Q$ and $y_0$ are called the target set and the initial state, respectively.
To avoid the triviality of these problems, we  often assume that
\begin{eqnarray}\label{y0-ball}
  y_0 \in L^2(\Omega)\setminus Q.
\end{eqnarray}

\begin{definition}\label{w-definition1.1}
(i) In Problem $(TP)^{M,y_0}_{Q}$, $T(M,y_0,Q)$ is called the minimal time;  $u\in\mathcal U^M$  is called an admissible control  if $y(\hat t;y_0,u)\in Q$ for some $\hat t \geq 0$; $u^*\in\mathcal U^M$ is called a minimal time control if  $T(M,y_0,Q)< +\infty$ and $y(T(M,y_0,Q);y_0,u^*)\in Q$. (ii) When $(TP)^{M,y_0}_{Q}$ has no admissible control,   $T(M,y_0,Q) \triangleq +\infty$. (iii) If the restrictions of all minimal time controls to $(TP)^{M,y_0}_{Q}$ over  $\big(0,T(M,y_0,Q)\big)$ are the same,
 then the minimal time control to this problem is said to be unique.
 (iv)  In Problem $(NP)^{T,y_0}_{Q}$, $N(T,y_0,Q)$ is called the minimal norm;  $v\in L^\infty(0,T;L^2(\Omega))$ is called an admissible control  if $\hat y(T;y_0,v)\in Q$; $v^*$ is called a minimal norm control  if
 $ \hat{y}(T;y_0,v^*)\in Q$  and $\|v^*\|_{L^\infty(0,T;L^2(\Omega))}=N(T,y_0,Q)$. (v)  The functions $M\rightarrow T(M,y_0,Q)$ and $T\rightarrow N(T,y_0,Q)$
are called the minimal time function and the minimal norm function, respectively.
\end{definition}

  In this paper, we aim to build up an equivalence between minimal time and minimal norm control problems in the sense of the following Definition~\ref{Def-1}.
  \begin{definition}\label{Def-1}
     Problems $(TP)^{M,y_0}_{Q}$ and $(NP)^{T,y_0}_{Q}$  are said to be equivalent if the following three conditions hold:
    (i) $(TP)^{M,y_0}_{Q}$ and $(NP)^{T,y_0}_{Q}$ have minimal time and minimal norm controls, respectively;
    (ii) The restriction of each minimal time control to $(TP)^{M,y_0}_{Q}$ over $(0,T)$ is a minimal norm control to $(NP)^{T,y_0}_{Q}$;
    (iii) The zero extension of each  minimal norm control to $(NP)^{T,y_0}_{Q}$ over  $\mathbb{R}^+$ is a minimal time control to $(TP)^{M,y_0}_{Q}$.
  \end{definition}

The  maim result of this paper is the following Theorem~\ref{theorem1.5-new}.

\begin{theorem}\label{theorem1.5-new}
Let $y_0 \in L^2(\Omega)$ and $Q \in \mathcal{F}$ satisfy (\ref{y0-ball}).
Write
 \begin{eqnarray}
(\mathcal GT)_{y_0,Q} &\triangleq&
\left\{ (M,T) \in [0,+\infty)\times (0,+\infty)\;:\; T=T(M,y_0,Q) \right\},
\label{def-GT}
\\
 (\mathcal KN)_{y_0,Q} &\triangleq&
 \left\{(M,T)  \in [0,+\infty)\times(0,+\infty) ~:~  M=0,\; N(T,y_0,Q)=0 \right\}.
 \label{def-KN}
\end{eqnarray}
Then the following conclusions are true:

\noindent(i) When $(M,T)\in (\mathcal GT)_{y_0,Q} \setminus (\mathcal KN)_{y_0,Q}$, problems $(TP)^{M,y_0}_{Q}$ and $(NP)^{T,y_0}_{Q}$  are equivalent and  the null controls (over $\mathbb{R}^+$ and $(0,T)$, respectively) are not the minimal time control and the minimal norm control to these two problems, respectively.

\noindent(ii) When $(M,T)\in (\mathcal KN)_{y_0,Q}$, problems $(TP)^{M,y_0}_{Q}$ and $(NP)^{T,y_0}_{Q}$ are equivalent and  the null controls (over $\mathbb{R}^+$ and $(0,T)$, respectively)
are the unique minimal time control and the unique minimal norm control to these two problems, respectively.

\noindent(iii) When $(M,T)\in [0,+\infty) \times (0,+\infty) \setminus\big( (\mathcal GT)_{y_0,Q} \cup (\mathcal KN)_{y_0,Q} \big)$, problems $(TP)^{M,y_0}_{Q}$ and $(NP)^{T,y_0}_{Q}$   are not equivalent.

\end{theorem}

Several notes are given in order.

(a)  Minimal time control problems and minimal norm control problems are two kinds of important optimal control problems in control theory. The equivalence between these two kinds of problems  plays an important role in the studies of these problems. To our best knowledge, in the existing  literatures on such equivalence (see, for instance, \cite{HOF,GL,WXZ,WangZhang,WZ,Y}), the target sets  are either the origin or balls, centered at the origin, in the state spaces. In \cite{HOF}, the author studied these two problems under an abstract framework where the target set is a point and controls enter the system globally. (This corresponds to the case that  $\omega=\Omega$.) The author
proved that the time optimality implies the norm optimality (see \cite[Theorem 2.1.2]{HOF}). It seems for us that the equivalence of these problems, where the target sets are arbitrary bounded closed convex sets with nonempty interiors in the state spaces and $\omega$ is a proper subset of $\Omega$, has not been touched upon. (At least, we did not find any such literature.)

 (b) In the case when the target set is $\{0\}$ and the initial state $y_0$ satisfies $y_0\neq 0$, the minimal norm and
 minimal time  functions  are  continuous and strictly decreasing from $\mathbb{R}^+$ onto $\mathbb{R}^+$, and they are inverses of each other (see  \cite[Theorem 2.1]{WZ}). With the aid of these  properties, the desired equivalence was   built up in \cite[Theorem 1.1]{WZ}. Besides, these properties imply that  $(\mathcal GT)_{y_0,Q}$ is connected and $(\mathcal KN)_{y_0,Q}=\emptyset$ (in the case that the target set is \{0\}).

    However, we will see from Theorem~\ref{theorem4-3} that for some $y_0$ and  $Q$ satisfying (\ref{y0-ball}), the minimal norm function is no longer  decreasing (correspondingly,  the minimal time function is not continuous). Indeed, it is proved in our paper that for some $ y_0$ and $Q$ satisfying (\ref{y0-ball}),
    the minimal norm function is not decreasing, $(\mathcal GT)_{y_0,Q}$ is not connected and $(\mathcal KN)_{y_0,Q}\neq\emptyset$ (see Theorem~\ref{theorem4-3}). These are the main differences of the current problems from those in \cite{WZ}. Such differences cause the main difficulty in the studies of the equivalence.

(c) We overcome the above-mentioned difficulty through building up  a new connection   between the minimal time function and the  minimal norm function (see Theorem \ref{Theorem-dis-TP-NP}). With the aid of this connection, as well as the continuity of the minimal norm function (see Theorem \ref{Theorem-Lip-NP}), we proved Theorem \ref{theorem1.5-new}. In the building of the above-mentioned new connection, we
     borrowed some idea from the classical raising sun lemma  (see, for instance,  Lemma 3.5 and Figure 5 on Pages 121-122 in \cite{EMStein}).

(d) About  works on  minimal time and minimal norm control problems, we would like to mention the references
\cite{ANJP, VB, carja-1, carja-2, HOF1, GY, KITO, KW1, KL1, KL2, LW, QL, SIMT, PZ, TWW, FTr, Wang, WX-1, WX-2, WCZ, WZheng, CZ-1, CZ-2, Guo-Zh} and the references therein.

The rest of the paper is organized as follows: Section 2 proves
 the main result. Section 3 provides an example where
  the minimal norm function is not decreasing, $(\mathcal GT)_{y_0,Q}$ is not connected and $(\mathcal KN)_{y_0,Q}\neq\emptyset$.

\section{Proof of the main theorem}
In this section, we first show some properties on the problem $(NP)^{T,y_0}_{Q}$; then study some
properties on the minimal norm and minimal time functions; finally give the proof of Theorem \ref{theorem1.5-new}.

\subsection{Some properties of  $(NP)^{T,y_0}_{Q}$}
In this subsection, we will present the existence of minimal norm controls and the bang-bang property for $(NP)^{T,y_0}_{Q}$.

\begin{theorem}\label{Lemma-exist-op-NP}
Given $T>0$, $y_0\in L^2(\Omega)$ and $Q\in \mathcal{F}$,  the problem $(NP)^{T,y_0}_{Q}$ has at least one minimal norm control.
\end{theorem}
\begin{proof}
Since $Q$ has a nonempty interior, it follows from the approximate controllability for the heat equation (see
\cite[Theorem 1.4]{FZ}) that $(NP)^{T,y_0}_{Q}$ has at least one admissible control. Then by the standard way (see, for instance, the proof of \cite[Lemma 1.1]{HOF1}), one can easily prove the existence of minimal norm controls to this problem. This ends the proof.
\end{proof}

The bang-bang property of $(NP)^{T,y_0}_{Q}$ can be directly derived from the Pontryagin maximum principle and the unique continuation for heat equations built up in \cite{Lin} (see also \cite{AWZ} and \cite{PZ}). Since we do not find the exact references on its proof, for the sake of the completeness of the paper, we give the proof here.

\begin{theorem}\label{Lemma-bang-bang-NP}
For each  $T>0$, $y_0\in L^2(\Omega)$ and $Q \in \mathcal{F}$, the problem $(NP)^{T,y_0}_{Q}$ holds the bang-bang property, i.e., every minimal norm control $v^*$ to $(NP)^{T,y_0}_{Q}$ satisfies that $\|v^*(t)\|=N(T,y_0,Q)$ a.e. $t\in(0,T)$.
\end{theorem}

\begin{proof}
Arbitrarily fix $T>0$, $y_0\in L^2(\Omega)$ and $Q \in \mathcal{F}$. There are only two possibilities on $N(T,y_0,Q)$: either $N(T,y_0,Q)=0$ or $N(T,y_0,Q)>0$. In the first case, it follows from Theorem~\ref{Lemma-exist-op-NP} and the definition of the minimal norm control (see (iv) in Definition~\ref{w-definition1.1}) that the null control is the unique minimal norm control to $(NP)^{T,y_0}_{Q}$. So this problem holds the bang-bang property in the first case.

In the second case that $  N(T,y_0,Q)>0$, we have that $e^{\Delta T}y_0 \notin Q$. Let
\begin{eqnarray}\label{bangbang-0}
 \mathcal A_{T} \triangleq \big\{ \hat y(T;y_0,v)\in L^2(\Omega) ~:~ \|v\|_{L^\infty(0,T;L^2(\Omega))}\leq N(T,y_0,Q)\big\}.
\end{eqnarray}
We claim that
\begin{eqnarray}\label{bangbang-1}
    \mathcal A_{T}  \cap  Q \neq \emptyset \;\;\mbox{and}\;\;
    \mathcal A_{T}  \cap  Q \subset \partial Q,
\end{eqnarray}
where $\partial Q$ denotes the boundary of $Q$.
In fact, by Theorem \ref{Lemma-exist-op-NP}, $(NP)^{T,y_0}_{Q}$ has a minimal norm control $v^*_1$ satisfying that
$\hat y(T;y_0,v^*_1)\in Q$ and  $\|v^*_1\|_{L^\infty(0,T;L^2(\Omega))}= N(T,y_0,Q)$. These, along with (\ref{bangbang-0}), imply that $\hat y(T;y_0,v^*_1)\in\mathcal A_{T}  \cap  Q$, which leads to
  the first conclusion in (\ref{bangbang-1}). We now  show the second conclusion in (\ref{bangbang-1}). By contradiction,  suppose that it were not true. Then, there would be $v_2\in L^{\infty}(0,T;L^2(\Omega))$ so that
\begin{eqnarray}\label{bangbang-3}
  \hat y(T;y_0,v_2) \in \mbox{Int}\; Q\;\; \mbox{and}\;\; \|v_2\|_{L^{\infty}(0,T;L^2(\Omega))} \leq N(T,y_0,Q),
\end{eqnarray}
where $\mbox{Int}\; Q$ denotes the interior of $Q$. Since $e^{\Delta T}y_0 \notin Q$, it follows from the first conclusion in (\ref{bangbang-3}) that $v_2$ is non-trivial, i.e., $v_2\neq 0$. By making use of the first conclusion in (\ref{bangbang-3}) again, we can choose  $\lambda\in(0,1)$, with $(1-\lambda)$ small enough, so that
$\hat y(T;y_0,\lambda v_2)\in  Q$.
Thus,  $\lambda v_2$ is an admissible control to $(NP)^{T,y_0}_{Q}$. This, along with the optimality of $N(T,y_0,Q)$, yields that $ N(T,y_0,Q)\leq \lambda\|v_2\|_{L^\infty(0,T;L^2(\Omega))}$. From this, the second conclusion in (\ref{bangbang-3}) and the non-triviality of $v_2$, we are led to  a contradiction. So (\ref{bangbang-1}) is true.

Since both $\mathcal{A}_T$ and $Q$ are convex sets and $\mbox{Int}\; Q\neq\emptyset$, by (\ref{bangbang-1}), we can apply the Hahn-Banach separation theorem to find $\eta^*\in L^2(\Omega)\setminus\{0\}$ so that
\begin{equation}\label{WWanggs2.4}
\sup_{z\in \mathcal A_{T}} \langle z,\eta^* \rangle
= \inf_{w\in Q}  \langle w,\eta^* \rangle.
\end{equation}
Let  $v^*$ be a minimal norm control  to $(NP)^{T,y_0}_{Q}$. Then, we have that
\begin{eqnarray}\label{bangbang-11}
\hat y(T;y_0,v^*)\in Q \;\;\mbox{and}\;\; \|v^*\|_{L^\infty(0,T;L^2(\Omega))}=N(T,y_0,Q),
\end{eqnarray}
from which and (\ref{bangbang-0}), it follows that $\hat y(T;y_0,v^*)\in \mathcal A_{T}  \cap  Q$. This, together with (\ref{WWanggs2.4}), yields  that
$ \max_{z\in \mathcal A_{T}} \langle z,\eta^* \rangle = \langle \hat y(T;y_0,v^*),\eta^* \rangle$.
By this and  (\ref{bangbang-0}), one can easily check that
\begin{eqnarray*}
 \max_{\|v\|_{L^\infty(0,T;L^2(\Omega))}\leq N(T,y_0,Q)} \int_0^T \langle v(t),\chi_\omega e^{\Delta(T-t)}\eta^* \rangle \mathrm dt= \int_0^T \langle v^*(t),\chi_\omega e^{\Delta(T-t)}\eta^* \rangle \mathrm dt.
\end{eqnarray*}
This, along with  the second conclusion in (\ref{bangbang-11}), yields that
\begin{eqnarray}\label{bangbang-12}
 \max_{\|v\|\leq N(T,y_0,Q)}  \langle v,\chi_\omega e^{\Delta(T-t)}\eta^* \rangle
 =  \langle v^*(t),\chi_\omega e^{\Delta(T-t)}\eta^* \rangle \;\;\mbox{ for a.e. }\;\; t\in (0,T).
\end{eqnarray}
Since $\eta^*\neq 0$, it follows from the unique continuation for heat equations (see \cite{Lin}) that
$\chi_\omega e^{\Delta(T-t)}\eta^*\neq 0$ for each $ t\in(0,T)$. This, along with (\ref{bangbang-12}), indicates that
$\|v^*(t)\|=N(T,y_0,Q)$  for a.e. $t\in(0,T)$. So $(NP)^{T,y_0}_{Q}$  holds the bang-bang property in the second case. This ends the proof.
\end{proof}
\begin{Remark}\label{wangremark2.3}
From the bang-bang property, one can easily show the uniqueness of the minimal norm control to $(NP)^{T,y_0}_{Q}$. In our studies, we will not use this uniqueness.
\end{Remark}

\subsection{Properties on the minimal norm and minimal time functions}

Let $y_0 \in L^2(\Omega)$ and $Q \in \mathcal {F}.$ For each $M \geq 0$, we define
\begin{eqnarray}\label{infty-emptyset}
 \mathcal{J}_M\triangleq \big\{t\in \mathbb R^+~:~N(t,y_0,Q)\leq M\big\}.
\end{eqnarray}
We agree that
\begin{eqnarray}\label{infty-emptyset-1}
 \inf \mathcal{J}_M \triangleq \inf \{t\;:\; t\in\mathcal{J}_M\} \triangleq +\infty,
 \;\;\mbox{when}\;\;
 \mathcal{J}_M=\emptyset.
\end{eqnarray}
The following theorem presents a connection between the minimal time and the minimal norm functions. Such connection plays an important role in our studies.
\begin{theorem}\label{Theorem-dis-TP-NP}
Let $y_0\in L^2(\Omega)$ and $Q \in \mathcal{F}$ satisfy (\ref{y0-ball}). Let $\mathcal{J}_M$, with $M\geq 0$,  be defined by (\ref{infty-emptyset}).
Then
\begin{eqnarray}\label{rearrange-TM-NT}
 T(M,y_0,Q)=\inf \mathcal{J}_M \;\;\mbox{for all}\;\; M \geq 0.
\end{eqnarray}
\end{theorem}
\begin{proof}
Arbitrarily fix  $y_0\in L^2(\Omega)$ and $Q \in \mathcal{F}$ satisfying (\ref{y0-ball}). Let $M\geq 0$.  Then, by  (\ref{y0-ball}), we see that
\begin{equation}\label{T-9-12-1}
  y(0; y_0, u) \notin Q \;\; \mbox{for all}\;\; u \in \mathcal{U}^M.
\end{equation}

In the case that $ \mathcal{J}_M = \emptyset $,   we first claim that
\begin{eqnarray}\label{dis-TP-NP-1}
 y(t;y_0,u)\notin Q \;\;\mbox{for all}\;\; t >0 \;\mbox{and}\; u\in\mathcal{U}^M.
\end{eqnarray}
By contradiction,  suppose that it were not true. Then there would be  $\hat t > 0$ and  $u_1\in L^\infty(\mathbb R^+;L^2(\Omega))$ so that
\begin{eqnarray}\label{dis-TP-NP-2}
 y(\hat t;y_0,u_1)\in Q  \;\;\mbox{and}\;\; \|u_1\|_{L^\infty(\mathbb R^+;L^2(\Omega))}\leq M.
\end{eqnarray}
The first conclusion in (\ref{dis-TP-NP-2}) implies that $u_1|_{(0,\hat t)}$ is an admissible control to
$(NP)^{\hat t,y_0}_Q$. This, along with the optimality of $N(\hat t,y_0,Q )$ and the second conclusion in (\ref{dis-TP-NP-2}), yields that
$N(\hat t,y_0,Q )\leq \|u_1\|_{L^\infty(0,\hat t;L^2(\Omega))}\leq M$,
which, along with (\ref{infty-emptyset}), shows that
 $\hat t\in \mathcal{J}_M$. This leads to a contradiction, since we are in the case that  $\mathcal{J}_M = \emptyset$. So (\ref{dis-TP-NP-1}) is true. Now, from (\ref{T-9-12-1}) and (\ref{dis-TP-NP-1}), we see that $(TP)_Q^{M,y_0}$ has no admissible control. Thus, it follows by (ii) of Definition~\ref{w-definition1.1} that
$T(M,y_0,Q)= +\infty$.
 This, together with (\ref{infty-emptyset-1}), leads to (\ref{rearrange-TM-NT}) in the case that $\mathcal{J}_M = \emptyset$.

We next consider the  case that $\mathcal{J}_M \neq \emptyset$. Arbitrarily take $\hat t\in\mathcal{J}_M$. Then by (\ref{infty-emptyset}), we see that
$N(\hat t,y_0,Q )\leq M$. Meanwhile, according to  Theorem \ref{Lemma-exist-op-NP},
$(NP)^{\hat t,y_0}_{Q}$ has a minimal norm control $v_{\hat t}$. Write $\hat v_{\hat t}$ for the zero extension of
 $v_{\hat t}$ over $\mathbb{R}^+$. One can easily check that
\begin{eqnarray}\label{dis-TP-NP-3}
  y(\hat t;y_0,\hat v_{\hat t})\in Q \;\;\mbox{and}\;\;
 \|\hat v_{\hat t}\|_{L^\infty(\mathbb{R}^+;L^2(\Omega))}=N(\hat t,y_0,Q)\leq M.
\end{eqnarray}
From (\ref{dis-TP-NP-3}),
we see that $\hat v_{\hat t}$ is an admissible control to $(TP)^{M,y_0}_{Q}$, which drives the solution to $Q$ at time $\hat t$. This,
along with
   the optimality of $T(M,y_0,Q)$, yields  that
 $T(M,y_0,Q)\leq \hat t$.
 Since $\hat t$ was arbitrarily taken from the set $\mathcal{J}_M$, the above implies that
\begin{eqnarray}\label{dis-TP-NP-5}
 T(M,y_0,Q)\leq \inf \mathcal{J}_M.
\end{eqnarray}

We now prove the reverse of (\ref{dis-TP-NP-5}). Define
\begin{eqnarray}\label{dis-TP-NP-6-0}
\mathcal T_{(M ,y_0,Q)}\triangleq \big\{ t \geq 0 ~:~ \exists\,u\in\mathcal U^M  \mbox{ s.t. } y(t;y_0,u)\in Q \big\} .
\end{eqnarray}
From (\ref{dis-TP-NP-3}) and (\ref{T-9-12-1}), it follows that
$
  \mathcal T_{(M ,y_0,Q)} \neq \emptyset
$
and
$ 0 \notin \mathcal T_{(M ,y_0,Q)}.
$
Then, given  $\tilde t\in \mathcal T_{(M ,y_0,Q)}$, with $\tilde t > 0$, there is a control
 $u_{\tilde t}$ in $ L^\infty(\mathbb R^+;L^2(\Omega))$ so that
\begin{eqnarray}\label{dis-TP-NP-6}
 y(\tilde t;y_0,u_{\tilde t})\in Q
  \;\;\mbox{and}\;\;
 \|u_{\tilde t}\|_{L^\infty(\mathbb R^+;L^2(\Omega))}\leq M.
\end{eqnarray}
The first conclusion in (\ref{dis-TP-NP-6}) implies that $u_{\tilde t}|_{(0,\tilde t)}$ is an admissible control to
$(NP)^{\tilde t,y_0}_{Q}$. This, along with the optimality of $N(\tilde t,y_0,Q )$ and the second conclusion in (\ref{dis-TP-NP-6}), yields that
$N(\tilde t,y_0,Q )\leq \|u_{\tilde t}|_{(0,\tilde t)}\|_{L^\infty(0, \tilde t; L^2(\Omega))}\leq M$, which, along with (\ref{infty-emptyset}), yields  that $\tilde t \in \mathcal{J}_M$. Hence, $ \inf \mathcal{J}_M \leq \tilde t$.
Since $0 \notin \mathcal T_{(M ,y_0,Q)}$, and because
$\tilde{t}>0$ was arbitrarily taken from  $\mathcal T_{(M ,y_0,Q)}$, we have that $ \inf \mathcal{J}_M \leq \inf \mathcal T_{(M ,y_0,Q)}$. This, along with (\ref{dis-TP-NP-6-0}) and  (\ref{TP}), shows  that
$\inf \mathcal{J}_M \leq T(M ,y_0,Q)$, which, together with (\ref{dis-TP-NP-5}), leads to (\ref{rearrange-TM-NT}) in the case that
$\mathcal{J}_M \neq \emptyset$.

In summary, we end the proof of this theorem.
\end{proof}

\begin{Remark}
For better understanding of Theorem~\ref{Theorem-dis-TP-NP}, we explain it  with the aid of Figure 2.1, where the curve denotes the graph of the minimal norm function.
Suppose that the minimal norm function is continuous over $\mathbb R^+$ (which will be proved in Theorem \ref{Theorem-Lip-NP}).
A beam (which is parallel to the $t$-axis and has the distance $M$ with the $t$-axis) moves from the left to the right. The first time point at which this beam reaches the curve is the minimal time to $(TP)^{M,y_0}_{Q}$. Thus, we  can  treat Theorem~\ref{Theorem-dis-TP-NP} as a ``falling sun theorem" (see, for instance, rasing sun lemma---Lemma 3.5 and Figure 5 on Pages 121-122 in \cite{EMStein}): If one thinks of the sun falling down the west (at the left) with the rays of light parallel to the $t$-axis, then the points $\big(T(M,y_0,Q),M\big)$, with $M\geq 0$,  are precisely the points which are in the bright part on the curve. (These points constitutes the part outside bold on the curve in  Figure 2.1.)
\begin{figure}[htp]
\begin{center}
  \includegraphics[width=4in]{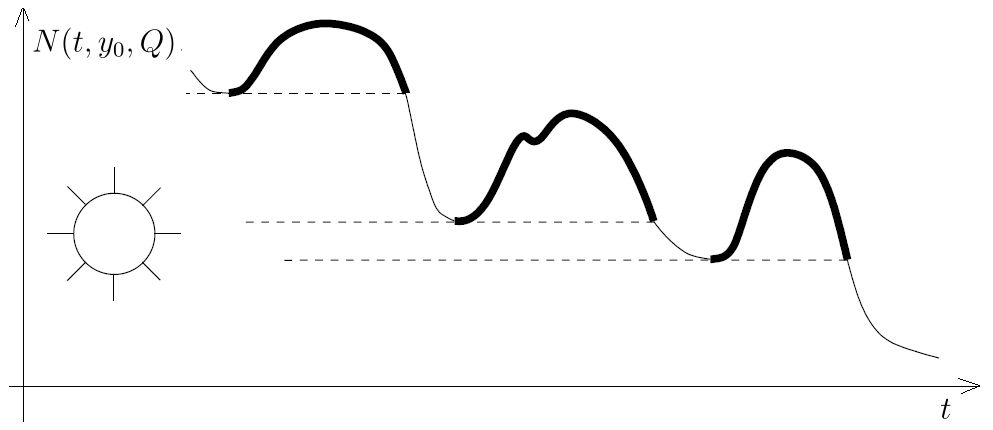}\\
  \caption{Falling sum theorem}
  \end{center}
\end{figure}
\end{Remark}

The following result mainly concerns with the continuity of the minimal norm function.

\begin{theorem}\label{Theorem-Lip-NP}
  Given $y_0\in L^2(\Omega)$ and $Q \in \mathcal{F}$, the minimal norm function
   $t \rightarrow N(t,y_0,Q )$
   is locally Lipschitz continuous over $\mathbb{R}^+$. If further assume that $y_0$ and $Q$  satisfy (\ref{y0-ball}), then $\lim_{t\rightarrow 0^+} N(t,y_0,Q ) = +\infty$.
\end{theorem}

\begin{proof}
We divide the proof into following three steps:

\textit{Step 1. To show that for each $y_0\in L^2(\Omega)$, $Q \in \mathcal{F}$ and each $\delta>0$, there is $C_1(\Omega,\omega,Q,\delta)>0$ so that
\begin{eqnarray}\label{NT-bdd}
 N(T,y_0,Q)\leq C_1(\Omega,\omega,Q,\delta)(\|y_0\|+1)
 \;\;\mbox{for all}\;\;
  T>\delta
\end{eqnarray}}
  $\quad~\,$Arbitrarily fix $y_0\in L^2(\Omega)$, $Q \in \mathcal{F}$ and $\delta>0$. Since the heat equation holds the approximate controllability (see \cite[Theorem 1.4]{FZ}) and the $L^\infty$-null controllability  (see, for instance, \cite[Proposition 3.2]{FZ}), there is $v_\delta$ (only depending on $\delta$, $Q$, $\Omega $ and $\omega$)
 and $v_\delta^\prime$ so that
  \begin{eqnarray}\label{Property-NT-2}
 \hat y(\delta;0,v_\delta)\in Q,\;\; \hat y(\delta;y_0,v_\delta^\prime)=0
 \;\;\mbox{and}\;\;
     \|v_\delta^\prime\|_{L^\infty(0,\delta;L^2(\Omega))} \leq C(\Omega,\omega,\delta)\|y_0\|.
 \end{eqnarray}
 Still write   $v_\delta^\prime$ for its zero extension over $\mathbb{R}^+$. Arbitrarily fix $T>\delta$. Define a control $\widehat v_T$ over $(0,T)$ as follow:
 \begin{eqnarray}\label{NP}
\widehat v_T(t)\triangleq
\left\{\begin{array}{ll}
         v_\delta^\prime(t),~t\in(0,T-\delta], \\
          v_\delta^\prime(t)+v_\delta(t-T+\delta),  ~t\in (T-\delta,T).
       \end{array}
\right.
\end{eqnarray}
 From (\ref{NP}) and  the last conclusion in  (\ref{Property-NT-2}), we find that
 \begin{eqnarray}\label{Property-NT-4}
  \|\widehat v_T\|_{L^\infty(0,T;L^2(\Omega))} \leq C(\Omega,\omega,\delta)\|y_0\|+\|v_\delta\|_{L^\infty(0,\delta;L^2(\Omega))}.
 \end{eqnarray}
Meanwhile,  from (\ref{NP}) and the first two conclusions in (\ref{Property-NT-2}),
we see that
 \begin{eqnarray*}
  \hat y(T;y_0,\widehat v_T) &=&   e^{\Delta T}y_0
                                +  \int_0^\delta e^{\Delta(T-t)}\chi_\omega v_\delta^\prime(t)\mathrm dt
                                +     \int_{T-\delta}^T e^{\Delta(T-t)}\chi_\omega v_\delta(t-T+\delta)\mathrm dt \nonumber\\
  &=&  e^{\Delta(T-\delta)}\hat y(\delta;y_0,v_\delta^\prime)
      +  \int_{0}^{\delta} e^{\Delta(\delta-t)}\chi_\omega v_\delta(t)\mathrm dt   = \hat y(\delta;0,v_\delta)\in Q,
 \end{eqnarray*}
 which shows  that $\widehat v_T$ is an admissible control to $(NP)^{T,y_0}_{Q}$. This, along with  the optimality of $N(T,y_0,Q)$ and (\ref{Property-NT-4}), indicates that
 \begin{eqnarray*}
  N(T,y_0,Q)\leq \|\widehat v_T\|_{L^\infty(0,T;L^2(\Omega))}
  \leq C(\Omega,\omega,\delta)\|y_0\|+\|v_\delta\|_{L^\infty(0,\delta;L^2(\Omega))},
 \end{eqnarray*}
 which leads to (\ref{NT-bdd}). Here we used the fact that $v_\delta$ only depends on $\delta$, $Q$, $\Omega $ and $\omega$.

\textit{Step 2. To show that for each $y_0\in L^2(\Omega)$, $Q\in \mathcal{F}$ and each  triplet $(\delta,T_1,T_2)$, with $0<\delta<T_2/2<T_1<T_2$, there exists a constant $C_2(\Omega,\omega,Q,\delta)>0$ so that
\begin{eqnarray}\label{Step2-conclu}
 |N(T_1,y_0,Q)-N(T_2,y_0,Q)| \leq  C_2(\Omega,\omega,Q,\delta)(\|y_0\|+1) |T_1-T_2|
\end{eqnarray}}
$\quad~\,$Arbitrarily fix $y_0\in L^2(\Omega)$, $Q\in \mathcal{F}$ and $(\delta,T_1,T_2)$ as required.
To prove (\ref{Step2-conclu}), it suffices to show that for some $C_2(\Omega,\omega,Q,\delta)>0$,
\begin{eqnarray}\label{Property-NT-20}
   N(T_1,y_0,Q) - N(T_2,y_0,Q)  \leq  C_2(\Omega,\omega,Q,\delta)(\|y_0\|+1)(T_2-T_1);
  \end{eqnarray}
\begin{eqnarray}\label{Property-NT-21}
   N(T_2,y_0,Q) - N(T_1,y_0,Q)  \leq  C_2(\Omega,\omega,Q,\delta)(\|y_0\|+1)(T_2-T_1).
\end{eqnarray}
To show (\ref{Property-NT-20}), we first note that $(NP)^{T_2,y_0}_{Q}$ has a minimal norm control  $v_{T_2}^*$
  (see  Theorem \ref{Lemma-exist-op-NP}). Thus,
\begin{eqnarray}\label{Property-NT-2-1}
 \hat y(T_2;y_0,v_{T_2}^*)\in Q
 \;\;\mbox{and}\;\;
  \| v_{T_2}^*\|_{L^\infty(0,T_2;L^2(\Omega))}=N(T_2,y_0,Q).
\end{eqnarray}
We set
\begin{eqnarray}\label{Property-NT-2-2}
  z_{T_1}\triangleq \hat y(3T_2/4;0,\chi_{(0,T_2-T_1)}v_{T_2}^*).
\end{eqnarray}
Since $T_2/4>\delta/4$, it follows from the $L^{\infty}$-null controllability for the heat equation (see \cite[Proposition 3.2]{FZ}) that there is  $f_1\in L^\infty(0,T_2/4;L^2(\Omega))$, with $f_1=0$ over $(\delta/4, T_2/4)$,  so that
\begin{eqnarray}\label{Property-NT-2-3}
0=\hat y(\delta/4;z_{T_1},f_1)=\hat y(T_2/4;z_{T_1},f_1);
\end{eqnarray}
\begin{eqnarray}\label{wanGG2.26}
 \|f_1\|_{L^\infty(0,T_2/4;L^2(\Omega))}\leq C(\Omega,\omega,\delta/4)\|z_{T_1}\|.
\end{eqnarray}
Since $T_2-T_1< 3T_2/4$,  from (\ref{Property-NT-2-2}) and  (\ref{Property-NT-2-3}),
we can easily check   that
\begin{eqnarray}\label{Property-NT-2-3-1}
 \hat y(T_2;0,\chi_{(0,T_2-T_1)}v_{T_2}^*) &=& e^{\Delta T_2/4} z_{T_1}=-\int_0^{T_2/4} e^{\Delta(T_2/4-t)} \chi_\omega f_1(t) \mathrm dt.
\end{eqnarray}

We next set
\begin{eqnarray}\label{Property-NT-2-2-1}
 w_{T_1}\triangleq e^{\Delta (T_1-T_2/4)}(e^{\Delta (T_2-T_1)} y_0-y_0).
\end{eqnarray}
Since $T_2/4>\delta/4$, by \cite[Proposition 3.2]{FZ}, we find that there is  $f_2\in L^\infty(0,T_2/4;L^2(\Omega))$, with $f_2=0$ over $(\delta/4,T_2/4)$,  so that
\begin{eqnarray}\label{Property-NT-2-4}
0=\hat y(\delta/4;w_{T_1},f_2)=\hat y(T_2/4;w_{T_1},f_2);
\end{eqnarray}
\begin{eqnarray}\label{GSwang2.29}
\|f_2\|_{L^\infty(0,T_2/4;L^2(\Omega))}\leq C(\Omega,\omega,\delta/4)\|w_{T_1}\|.
\end{eqnarray}
 From (\ref{Property-NT-2-2-1}) and  (\ref{Property-NT-2-4}), we see  that
\begin{eqnarray}\label{Property-NT-2-3-2}
 e^{\Delta T_2}y_0-e^{\Delta T_1}y_0 &=& e^{\Delta T_2/4} w_{T_1}=-\int_0^{T_2/4} e^{\Delta(T_2/4-t)} \chi_\omega f_2(t) \mathrm dt.
\end{eqnarray}

Now we define a control $f_3$ over $(0,T_1)$ by
\begin{eqnarray}\label{Property-NT-2-5}
  f_3(t)\triangleq  \left\{\begin{array}{ll}
                     v_{T_2}^*(t-T_1+T_2),&t\in (0,T_1-T_2/4),\\
                     v_{T_2}^*(t-T_1+T_2)-f_1(t-T_1+T_2/4)  & \\
                     -f_2(t-T_1+T_2/4),&t\in (T_1-T_2/4,T_1).
                    \end{array}
                    \right.
 \end{eqnarray}
Two observations are given in order: First, by (\ref{Property-NT-2-5}),  (\ref{wanGG2.26}) and  (\ref{GSwang2.29}), we see that
\begin{eqnarray*}
 \|f_3\|_{L^\infty(0,T_1;L^2(\Omega))}  \leq \|v_{T_2}^*\|_{L^\infty(T_2-T_1,T_2;L^2(\Omega))} + C(\Omega,\omega,\delta/4)\big(\|z_{T_1}\|+\|w_{T_1}\| \big).
\end{eqnarray*}
Second,  by (\ref{Property-NT-2-5}), (\ref{Property-NT-2-3-1}), (\ref{Property-NT-2-3-2}) and (\ref{Property-NT-2-2}),
 after some computations, one can easily check that $\hat y(T_1;y_0,f_3)=\hat y(T_2;y_0,v^*_{T_2})$, which, along with the first conclusion in (\ref{Property-NT-2-1}), yields that
 $f_3$ is an admissible control to $(NP)^{T_1,y_0}_{Q}$.
 These two observations,  together with the optimality of
  $N(T_1,y_0,Q)$, indicate that
  \begin{eqnarray*}
   N(T_1,y_0,Q)\leq \|v_{T_2}^*\|_{L^\infty(0,T_2;L^2(\Omega))} + C(\Omega,\omega,\delta/4)\big(\|z_{T_1}\|+\|w_{T_1}\| \big).
  \end{eqnarray*}
  This, along with the second conclusion in (\ref{Property-NT-2-1}), (\ref{Property-NT-2-2}) and (\ref{Property-NT-2-2-1}), implies that
  \begin{eqnarray}\label{Property-NT-20-0}
    N(T_1,y_0,Q)   &\leq& N(T_2,y_0,Q)   + C(\Omega,\omega,\delta/4)  \,\big[ (T_2-T_1)N(T_2,y_0,Q)
     \nonumber\\
    & & +\|e^{\Delta (T_1-T_2/4)}(e^{\Delta (T_2-T_1)} y_0-y_0)\| \big] .
  \end{eqnarray}
  Meanwhile, since $T_1-T_2/4>T_2/4> \delta/2$, it follows that
  \begin{eqnarray}\label{Property-NT-20-1}
   & & \|e^{\Delta (T_1-T_2/4)}(e^{\Delta (T_2-T_1)} y_0-y_0)\| \nonumber\\
   &\leq& \|e^{\Delta T_2/4}(e^{\Delta (T_2-T_1)} y_0-y_0)\|
   = \big\|\int_0^{T_2-T_1}\Delta e^{\Delta s} (e^{\Delta T_2/4} y_0)\mathrm ds\big\|
   \nonumber\\
   &\leq& (T_2-T_1)\|\Delta e^{\Delta T_2/4} y_0\|\leq 4(T_2-T_1)\|y_0\|/T_2\leq 2(T_2-T_1)\|y_0\|/\delta.~~~~~
  \end{eqnarray}
  From (\ref{Property-NT-20-0}), (\ref{Property-NT-20-1}) and (\ref{NT-bdd}),
  we obtain (\ref{Property-NT-20}).

To show (\ref{Property-NT-21}), we note that
    $(NP)^{T_1,y_0}_{Q}$ has a minimal norm control  $v_{T_1}^*$ (see
     Theorem \ref{Lemma-exist-op-NP}). Thus,
\begin{eqnarray}\label{Property-NT-2-2-10}
 \hat y(T_1;y_0,v_{T_1}^*)\in Q
 \;\;\mbox{and}\;\;
  \| v_{T_1}^*\|_{L^\infty(0,T_1;L^2(\Omega))}=N(T_1,y_0,Q).
\end{eqnarray}
We define a control $f_4$ by
\begin{eqnarray}\label{Property-NT-2-2-5}
  f_4(t)\triangleq  \left\{\begin{array}{ll}
                      0,& t\in (0,T_2-T_1),\\
                     v_{T_1}^*(t-T_2+T_1),&t\in (T_2-T_1,3T_2/4),\\
                     v_{T_1}^*(t-T_2+T_1)+f_2(t-3T_2/4),&t\in (3T_2/4,T_2),
                    \end{array}
                    \right.
 \end{eqnarray}
  where $f_2$ is given by (\ref{Property-NT-2-4}). By (\ref{Property-NT-2-2-5}) and (\ref{Property-NT-2-3-2}),
   after some direct computations, we find that $\hat y(T_2; y_0,f_4)=\hat y(T_1;y_0,v^*_{T_1})$, which, along with the first conclusion in (\ref{Property-NT-2-2-10}), shows that  $f_4$ is an admissible control to $(NP)^{T_2,y_0}_{Q}$. This, together with the optimality of
  $N(T_2, y_0, Q )$ and (\ref{Property-NT-2-2-5}), yields that
  \begin{eqnarray*}
   N(T_2,y_0,Q)\leq \|f_4\|_{L^\infty(0,T_2;L^2(\Omega))} \leq \|v^*_{T_1}\|_{L^\infty(0,T_1;L^2(\Omega))} +\|f_2\|_{L^\infty(0,T_2/4;L^2(\Omega))}.
  \end{eqnarray*}
  Then we see from the second conclusion in   (\ref{Property-NT-2-2-10}) and (\ref{GSwang2.29}) that
  \begin{eqnarray*}
   N(T_2,y_0,Q)\leq N(T_1,y_0,Q) +C(\Omega,\omega,\delta/4)\|w_{T_1}\|.
  \end{eqnarray*}
  This, along with (\ref{Property-NT-2-2-1}) and (\ref{Property-NT-20-1}), leads to (\ref{Property-NT-21}).

\textit{Step 3. To show that $\lim_{t \rightarrow 0^+} N(t,y_0,Q)= +\infty$, when $(y_0,Q)\in L^2(\Omega)\times\mathcal{F}$  verifies (\ref{y0-ball}) }

By contradiction,
we suppose that it were not true. Then  there would be   $(y_0, Q)\in L^2(\Omega)\times \mathcal{F}$, with  (\ref{y0-ball}); and
 $\{t_n\}\subset\mathbb R^+$, with $\lim_{n\rightarrow +\infty} t_n=0$, so that
\begin{equation}\label{step3-1}
 \sup_{n\in \mathbb N^+} N(t_n,y_0,Q ) \leq C  \;\;\mbox{for some}\;\; C>0.
\end{equation}
By  Theorem \ref{Lemma-exist-op-NP}, we find that for each $n\in\mathbb N^+$, $(NP)^{t_n,y_0}_{Q}$ has a minimal norm control $v_n$ satisfying that
 \begin{eqnarray}\label{Property-NT-36}
  \hat y(t_n;y_0,v_n)\in Q \;\;\mbox{ and }\;\; \|v_n\|_{L^\infty(0,T_n;L^2(\Omega))} =N(t_n,y_0,Q ).
 \end{eqnarray}
 Extend $v_n$ over $\mathbb R^+$ by setting it to be zero over $[t_n, +\infty)$ and still denote the extension  in the same manner.
 Since $\lim_{n\rightarrow +\infty} t_n=0$, it follows from the second conclusion in (\ref{Property-NT-36}) and  (\ref{step3-1})  that
 $  \chi_{(0,t_n)} v_n\rightarrow 0$  strongly in $L^2(\mathbb R^+;L^2(\Omega))$,
  as  $n\rightarrow +\infty$.
 This, together with the first conclusion in (\ref{Property-NT-36})
 and the fact that $\lim_{n\rightarrow +\infty} t_n=0$,
  yields that
 $ y_0=\lim_{n\rightarrow +\infty} \hat y(t_n;y_0,v_n)\in Q$,
  which contradicts (\ref{y0-ball}). Hence, the conclusion in Step 3 is true.

 In summary,   we end the proof of this theorem.
\end{proof}

 By Theorem \ref{Theorem-Lip-NP} and Theorem \ref{Theorem-dis-TP-NP}, we can prove the following  Proposition \ref{Proposition-N-T}, which will be used in the proof of Theorem~\ref{theorem1.5-new}.
\begin{proposition}\label{Proposition-N-T}
Let $y_0 \in L^2(\Omega)$ and $Q \in \mathcal{F}$  satisfy (\ref{y0-ball}). Then the following two conclusions are valid:\\
(i) For each $M \in [0, +\infty)$, it holds that  $T(M,y_0,Q)\in (0, +\infty]$. \\
(ii) For each $M \in [0, +\infty)$, with $T(M,y_0,Q)<+\infty$, it stands that
\begin{eqnarray}\label{rearrange-TM-NT-1}
 N(T(M,y_0,Q),y_0,Q)=M .
\end{eqnarray}

\end{proposition}

\begin{proof}
Arbitrarily fix $y_0$ and $Q$  satisfying (\ref{y0-ball}). We will prove (i)-(ii) one by one.

(i)
 By contradiction, we suppose that the conclusion (i) were not true. Then
 there would be $M \in [0, +\infty)$ so that
   $T(M,y_0,Q)=0$. This, along with  (\ref{TP}), yields that there  exists $\{\hat t_n\}\subset [0, +\infty)$ and $\{u_n\}\subset\mathcal U^M$ so that
\begin{equation}\label{NTM}
  \lim_{n\rightarrow +\infty} \hat t_n=0,\;\;
 y(\hat t_n;y_0,u_n)\in Q
 \;\;\mbox{and}\;\;
  \| u_n\|_{L^\infty(\mathbb R^+;L^2(\Omega))} \leq M.
\end{equation}
By the last inequality in (\ref{NTM}),
$\chi_{(0,\hat t_n)}u_n \rightarrow 0$ strongly in $L^2(\mathbb R^+;L^2(\Omega))$.
This, along with the first two conclusions in (\ref{NTM}), yields that
$y_0=\lim_{n\rightarrow +\infty}y(\hat t_n;y_0,u_n)\in Q$,
which contradicts (\ref{y0-ball}). So the conclusion (i) is true.

(ii) Arbitrarily fix $M \in [0, +\infty)$ so that
$T(M,y_0,Q)<+\infty$.
By the conclusion (i) in this proposition, we find that
$0<T(M,y_0,Q)<+\infty$.
This, along with (\ref{rearrange-TM-NT}) (see Theorem \ref{Theorem-dis-TP-NP}) and (\ref{infty-emptyset-1}),
indicates that $\mathcal{J}_M\neq\emptyset$. Thus, by (\ref{rearrange-TM-NT}), (\ref{infty-emptyset-1}) and (\ref{infty-emptyset}), there is a sequence $\{t_n\}\subset\mathbb R^+$ so that
\begin{eqnarray*}
 \lim_{n \rightarrow +\infty}t_n = T(M,y_0,Q)
 \;\;\mbox{and}\;\;
   N(t_n,y_0,Q )\leq M
   \;\;\mbox{for all}\;\;   n\in\mathbb N^+.
\end{eqnarray*}
Since $T(M,y_0,Q) \in  (0, +\infty)$, the above conclusions, together with the continuity of the minimal norm function at  $T(M,y_0,Q)\in (0, +\infty)$ (see Theorem \ref{Theorem-Lip-NP}), yield that
\begin{eqnarray}\label{TP-NP-0}
 N(T(M,y_0,Q),y_0,Q)\leq M .
\end{eqnarray}
We next prove (\ref{rearrange-TM-NT-1}). By contradiction, we suppose that it were not true. Then by (\ref{TP-NP-0}), we would have that $ N(T(M,y_0,Q),y_0,Q)< M$. This, along with  the continuity of the minimal norm function at $T(M,y_0,Q)$, yields  that there is $\delta_0\in \big(0,T(M,y_0,Q)\big)$ so that $N(T(M,y_0,Q)-\delta_0,y_0,Q)<M$. Then it follows from (\ref{infty-emptyset}) that $T(M,y_0,Q)-\delta_0 \in \mathcal{J}_M$, which  contradicts  (\ref{rearrange-TM-NT}). Hence (\ref{rearrange-TM-NT-1}) is true.

\vskip 2pt
In summary, we end the proof of this proposition.
\end{proof}

\subsection{Proof of Theorem~\ref{theorem1.5-new}}
Let $y_0 \in L^2(\Omega)$ and $Q \in \mathcal F$ satisfying (\ref{y0-ball}). We prove the conclusions (i)-(iii) one by one.

(i) Arbitrarily fix $(M,T)$ so that
\begin{eqnarray}\label{0608-new-1}
(M,T)\in (\mathcal GT)_{y_0,Q} \setminus (\mathcal KN)_{y_0,Q}.
\end{eqnarray}
Then it follows  from  (\ref{def-GT}) that
\begin{eqnarray}\label{0608-new-2}
 0<T=T(M,y_0,Q)<+\infty.
\end{eqnarray}

CLAIM ONE: The problems $(NP)^{T,y_0}_{Q}$ and $(TP)^{M,y_0}_{Q}$  have minimal norm and minimal time controls, respectively. Indeed, since $0<T<+\infty$ (see (\ref{0608-new-2})), it follows from Theorem \ref{Lemma-exist-op-NP} that $(NP)^{T,y_0}_{Q}$ has at least one minimal norm control. Meanwhile, since $T(M,y_0,Q)<+\infty$ (see (\ref{0608-new-2})), it follows from (ii) of Definition~\ref{w-definition1.1} that $(TP)^{M,y_0}_{Q}$ has at least one admissible control. Then by a standard way (see, for instance,  the proof of \cite[Lemma 1.1]{HOF1}), we can prove that $(TP)^{M,y_0}_{Q}$ has at least one minimal time control.

 CLAIM TWO: For an arbitrarily fixed minimal time control $u^*_1$ to $(TP)^{M,y_0}_{Q}$,
 $u^*_1|_{(0,T)}$ is
  a minimal norm control to $(NP)^{T,y_0}_{Q}$. Indeed, by the optimality of $u^*_1$ and (\ref{0608-new-2}), we have that
\begin{eqnarray}\label{proof-1-3}
 y(T; y_0,u_1^*)=y(T(M,y_0,Q);y_0,u_1^*)\in Q
 \;\;\mbox{and}\;\;
  \|u_1^*\|_{L^\infty(\mathbb R^+;L^2(\Omega))}\leq M.
\end{eqnarray}
By the first conclusion in (\ref{proof-1-3}), we see that $u_1^*|_{(0,T)}$ is an admissible control to $(NP)^{T,y_0}_{Q}$. This, along with the optimality of $N(T,y_0,Q)$ and the second conclusion in (\ref{proof-1-3}), yields that
\begin{equation}\label{wang3.9}
N(T,y_0,Q)\leq \|u_1^*|_{(0,T)}\|_{L^\infty(0,T;L^2(\Omega))}\leq  M.
\end{equation}
Meanwhile, by  (\ref{0608-new-2}), we can apply (ii) of Proposition \ref{Proposition-N-T} to find that
\begin{eqnarray}\label{proof-1-6-1}
N(T,y_0,Q)=N(T(M,y_0,Q),y_0,Q ) = M.
\end{eqnarray}
From (\ref{wang3.9}) and (\ref{proof-1-6-1}), we see that
\begin{eqnarray*}\label{proof-1-5}
\|u_1^*|_{(0,T)}\|_{L^\infty(0,T;L^2(\Omega))} = N(T,y_0,Q).
\end{eqnarray*}
Since $u_1^*|_{(0,T)}$ is an admissible control to $(NP)^{T,y_0}_{Q}$, the above shows that $u_1^*|_{(0,T)}$ is a minimal norm control to $(NP)^{T,y_0}_{Q}$.

CLAIM THREE: For an arbitrarily fixed minimal norm control  $v_1^*$ to $(NP)^{T,y_0}_{Q}$,
the zero extension of $v_1^*$ over $\mathbb{R}^+$, denoted by $\widetilde v_1^*$, is a minimal time control to $(TP)^{M,y_0}_{Q}$. Indeed, by the optimality of $v_1^*$, one can easily check that
\begin{eqnarray*}\label{proof-1-6}
  y(T;y_0,\widetilde v_1^*)\in Q
 \;\;\mbox{and}\;\;
 \|\widetilde v_1^*\|_{L^\infty(\mathbb{R}^+;L^2(\Omega))}=N(T,y_0,Q).
\end{eqnarray*}
From this, (\ref{0608-new-2}), (\ref{proof-1-6-1}) and (i) of Definition~\ref{w-definition1.1}, we find that  $\widetilde v_1^*$ is a minimal time control to $(TP)^{M,y_0}_{Q}$.

Now,  by the above three claims and Definition~\ref{Def-1}, we see that the problems $(TP)^{M,y_0}_{Q}$ and $(NP)^{T,y_0}_{Q}$ are equivalent.

Finally,  we claim that
\begin{eqnarray}\label{0611-aa}
 N(T,y_0,Q) \neq 0.
\end{eqnarray}
When  (\ref{0611-aa}) is proved, we can easily check that  the null control
 (defined on $(0,T)$) is not a minimal norm control to $(NP)^{T,y_0}_{Q}$.
  Then by the  equivalence of $(TP)^{M,y_0}_{Q}$ and $(NP)^{T,y_0}_{Q}$, we can easily prove that the null control
  (defined on $\mathbb{R}^+$) is not a minimal time control to $(TP)^{M,y_0}_{Q}$.

The remainder is to show (\ref{0611-aa}). By contradiction, suppose that (\ref{0611-aa}) were not true. Then we would have that  $N(T,y_0,Q) =0$. This, together with  (\ref{proof-1-6-1}), yields that $M=0$. From this and (\ref{def-KN}), we get that $(M,T)\in (\mathcal KN)_{y_0,Q}$, which contradicts (\ref{0608-new-1}). Therefore, (\ref{0611-aa}) is true.
 This ends the proof of the conclusion (i) in Theorem~\ref{theorem1.5-new}.

(ii) Without loss of generality, we can assume that
$(\mathcal KN)_{y_0,Q}  \neq \emptyset$.
Arbitrarily fix
\begin{eqnarray}\label{0608-new-ii-1}
(M,T)\in  (\mathcal KN)_{y_0,Q}.
\end{eqnarray}
By Definition~\ref{Def-1}, we see that in order to prove the conclusion (ii),  it suffices to show that the null controls (defined on $(0,T)$ and $\mathbb{R}^+$, respectively) are the unique minimal norm control and
the unique minimal time control to $(NP)^{T,y_0}_{Q}$ and $(TP)^{M,y_0}_{Q}$, respectively. To this end, we observe from  (\ref{0608-new-ii-1}) and  (\ref{def-KN}) that
\begin{eqnarray}\label{0608-new-ii-2}
 0<T<+\infty,\;
 M=0  \;\;\mbox{and}\;\;
 N(T,y_0,Q)=0.
\end{eqnarray}
By the first conclusion of (\ref{0608-new-ii-2}) and Theorem \ref{Lemma-exist-op-NP}, we see that $(NP)^{T,y_0}_{Q}$ has at least one minimal norm control. This, along with the last conclusion in (\ref{0608-new-ii-2}),  implies that   the null control (defined on $(0,T)$) is the unique minimal norm control to $(NP)^{T,y_0}_{Q}$. From this, it follows that
$\hat y(T;y_0,0)\in Q$, from which, one can easily check that the null control (defined on $\mathbb{R}^+$) is admissible for $(TP)^{M,y_0}_{Q}$. Then by a standard way (see, for instance,  the proof of \cite[Lemma 1.1]{HOF1}), we can prove that $(TP)^{M,y_0}_{Q}$ has at least one minimal time control. This, along with the second conclusion in (\ref{0608-new-ii-2}), indicates that  the null control is the unique minimal time control to $(TP)^{M,y_0}_{Q}$.
  This ends the proof of the conclusion (ii) in Theorem~\ref{theorem1.5-new}.

\vskip 5pt
(iii)
By contradiction,  suppose that the conclusion (iii) were not true. Then
 there would be a pair
\begin{eqnarray}\label{0608-new-iii-1}
(M,T)\in [0, +\infty)\times (0, +\infty) \setminus\big( (\mathcal GT)_{y_0,Q} \cup (\mathcal KN)_{y_0,Q} \big)
\end{eqnarray}
so that $(TP)^{M,y_0}_{Q}$ and $(NP)^{T,y_0}_{Q}$ are equivalent.
The key to get a contradiction is to prove that
\begin{eqnarray}\label{0608-new-iii-2}
 T(M,y_0,Q) = T.
\end{eqnarray}
When this  is proved, we see from
 (\ref{0608-new-iii-2}) and  (\ref{def-GT}) that
$ (M,T)\in  (\mathcal GT)_{y_0,Q}$. (Notice that $0<T<+\infty$ and $0\leq M<+\infty$.) This contradicts (\ref{0608-new-iii-1}).
Hence, the conclusion (iii) is true.

We now  prove (\ref{0608-new-iii-2}). Since $(TP)^{M,y_0}_{Q}$ and $(NP)^{T,y_0}_{Q}$ are equivalent,
two facts are derived from  Definition~\ref{Def-1}: First, $(NP)^{T,y_0}_{Q}$ has a minimal norm control
 $v_2^*$; Second, the zero extension of $v_2^*$ over $\mathbb{R}^+$, denoted by $\widetilde v_2^*$,  is a minimal time control to $(TP)^{M,y_0}_{Q}$. From these two facts, we can easily check that $ T(M,y_0,Q) \leq T$.
 This, along with (i) of Proposition \ref{Proposition-N-T}, shows that
  \begin{eqnarray}\label{0608-new-iii-4}
   0<T(M,y_0,Q) \leq T.
 \end{eqnarray}

 By contradiction, we suppose that (\ref{0608-new-iii-2}) were not true. Then by (\ref{0608-new-iii-4}) and (\ref{0608-new-iii-1}), we would have that
 \begin{eqnarray}\label{0608-new-iii-5}
  0<T(M,y_0,Q) < T < +\infty.
 \end{eqnarray}
  It follows from (\ref{0608-new-iii-5}) and (ii) of Definition~\ref{w-definition1.1} that $(TP)^{M,y_0}_{Q}$ has at least one admissible control. Then by a standard way (see, for instance,  the proof of \cite[Lemma 1.1]{HOF1}), we can prove that $(TP)^{M,y_0}_{Q}$ has a minimal time control $u^*$. Thus we have that
 \begin{eqnarray}\label{0608-new-iii-6}
 y(T(M,y_0,Q);y_0,u^*)\in Q
 \;\;\mbox{and}\;\;
  \|u^*\|_{L^\infty(\mathbb R^+;L^2(\Omega))}\leq M.
\end{eqnarray}
Arbitrarily take $\hat z \in L^2(\Omega)\setminus\{0\}$. Define
\begin{eqnarray*}
 \hat u_0^*(t)=
 \left\{\begin{array}{l}
 u^*(t),~t\in \big(0,T(M,y_0,Q)\big),\\
 0,~t\in [T(M,y_0,Q), +\infty);
 \end{array}
 \right.
 \;\;
 \hat u_M^*(t)=
 \left\{\begin{array}{l}
 u^*(t),~t\in \big(0,T(M,y_0,Q)\big),\\
 M \frac{\hat z}{\|\hat z\|},~t\in [T(M,y_0,Q), +\infty).
 \end{array}
 \right.
\end{eqnarray*}
From these and (\ref{0608-new-iii-6}), we see that $\hat u_0^*$ and $\hat u_M^*$ are minimal time controls to $(TP)^{M,y_0}_{Q}$. By the equivalence of  $(TP)^{M,y_0}_{Q}$ and $(NP)^{T,y_0}_{Q}$, we find that  $\hat u_0^*|_{(0,T)}$ and $\hat u_M^*|_{(0,T)}$ are minimal norm controls to $(NP)^{T,y_0}_{Q}$. This, along with  (\ref{0608-new-iii-5}) and Theorem \ref{Lemma-bang-bang-NP}, indicates that for a.e. $t\in
  \big(T(M,y_0,Q), T \big)$,
 \begin{eqnarray*}
   N(T,y_0,Q)=\|\hat u_0^*(t)\|=0 \;\;\mbox{and}\;\; N(T,y_0,Q)=\|\hat u_M^*(t)\|=M.
 \end{eqnarray*}
 From these, it follows that $N(T,y_0,Q)=M=0$.
  This, along with  (\ref{def-KN}), yields that
 $(M,T)\in  (\mathcal KN)_{y_0,Q}$,
which contradicts (\ref{0608-new-iii-1}). Thus, (\ref{0608-new-iii-2}) is true.
This ends the proof of the conclusion (iii) in Theorem~\ref{theorem1.5-new}.

In summary, we conclude that  the conclusions (i), (ii) and (iii) are true.
This completes the proof of Theorem~\ref{theorem1.5-new}.

\section{Further illustrations }

The aim of  this section is to construct an example where the minimal norm function $T\rightarrow N(T,y_0,Q)$ is not  decreasing, the set $(\mathcal KN)_{y_0,Q}$ is not empty and the set $(\mathcal GT)_{y_0,Q}$ is not connected.
 This example may help us to understand Theorem~\ref{theorem1.5-new} better.

\begin{theorem}\label{theorem4-3}
There exists $Q\in \mathcal F$ and $y_0 \in L^2(\Omega)\setminus Q$ so that
 the following propositions are true:\\
 (i) The function $T\rightarrow N(T, y_0, Q)$ is not   decreasing;\\
(ii) The set $(\mathcal KN)_{y_0,Q} \neq \emptyset$ and the set $(\mathcal GT)_{y_0,Q}$ is not connected.

\end{theorem}

To prove the above theorem, we need some preliminaries. The following lemma concerns with some kind of continuity of the map
$Q\rightarrow N(T,y_0,Q)$.

\begin{lemma}\label{theorem4-1}
Let $E\subset  L^2(\Omega)$ be a finite dimension subspace, with its orthogonal space $E^\perp$.  Suppose that  $\{S_n\}_{n=1}^{+\infty}\subset E$  is  an increasing sequence of bounded closed convex  subsets. Assume that  $ S \triangleq \cup_{n=1}^{+\infty} S_n$ verifies that
\begin{eqnarray}\label{theorem4-q-1}
\hat Q\triangleq \bar S \oplus B_{E^\perp} \in \mathcal F,
\end{eqnarray}
 where $\bar S$ and $B_{E^\perp}$ denote the closure of $S$ in $E$ and the closed unit ball in $E^\perp$, respectively. Let
 \begin{eqnarray}\label{wang3.2}
 Q_n\triangleq S_n \oplus B_{E^\perp}\;\;\mbox{for each}\;\;n\in\mathbb N^+.
 \end{eqnarray}
 Then for all $n$ large enough, $Q_n\in \mathcal{F}$; and for all $y_0\in L^2(\Omega)$ and for all  $a$ and $b$, with $0 < a < b < +\infty$,
\begin{eqnarray}\label{theorem4-1-a}
 N(T, y_{0}, Q_n)  \rightarrow N(T, y_0,\hat Q)  \mbox{ uniformly w.r.t. } T\in[a,b],\;\;\mbox{as}\;\; n\rightarrow\infty.
\end{eqnarray}
\end{lemma}
\begin{proof}
 By using \cite[Theorem 1.1.14]{Schneider} and (\ref{theorem4-q-1}) and the definition of $\mathcal F$, we see  that $\bar S$ has a nonempty interior in $E$.
 Then, by the finite dimensionality and the convexity of $S$, one can easily check
 that $S$ has a nonempty interior in $E$.
  Since $S=\cup_{n=1}^\infty S_n \subset E$, it follows from the Baire Category Theorem that $S_{N_0}$ has a nonempty interior in $E$ for some $N_0\in\mathbb N^+$. Thus, by the monotonicity of $\{S_n\}_{n=1}^\infty$ and the definition of
     $Q_n$,  there exists a closed ball  $ B_r(y_d)$  in $L^2(\Omega)$, centered at $y_d \in L^2(\Omega)$ and of radius $r>0$, so that
\begin{eqnarray}\label{0815-th3.2-yd}
 B_r(y_d) \subset Q_{N_0} \subset Q_n
 \;\;\mbox{for all}\;\;
 n\geq N_0.
\end{eqnarray}
From (\ref{wang3.2}) and (\ref{0815-th3.2-yd}), we see that when $n\geq N_0$,
$Q_n\in \mathcal{F}$.

Now we arbitrarily fix $y_0\in L^2(\Omega)$ and $0 < a < b < +\infty$. Then arbitrarily fix $T \in[a,b]$. The rest proof is divided into the following four steps.

  \textit{Step 1. To show that there exists $C \triangleq C(\Omega,b,y_d,r)>0$ (independent of $T\in [a,b]$),
  $\hat y_{0,T}\in L^2(\Omega)$  and
  $\hat y_d\in L^2(\Omega)$ (independent of $T\in [a,b]$) so that
\begin{eqnarray}\label{lemma4-1-3}
\|\hat y_d-y_d\| \leq r/2,~
\hat y_d = e^{\Delta T} \hat y_{0,T}
\;\;\mbox{and}\;\;
\|\hat y_{0,T} \| \leq C
  \end{eqnarray}}
$\quad~\,$Let $\{\lambda_j\}_{j=1}^{+\infty}$ be the family of all eigenvalues of $-\Delta$ with the zero Dirichlet boundary condition so that
\begin{equation}\label{lemma4-1-25}
  0< \lambda_1 \leq \lambda_2 \leq \cdots \;\;\mbox{and}\;\;
  \lim_{j\rightarrow +\infty}\lambda_j= \infty.
\end{equation}
Write  $\{e_j\}_{j=1}^{+\infty}$ for the family of  the corresponding normalized eigenvectors. Since $ y_d \in L^2(\Omega)$, we can choose a positive integer $K\triangleq K(y_d,r)$  large enough so that $\Sigma_{j=K+1}^\infty \langle  y_d, e_j \rangle^2 \leq  r^2/4$. Let $\hat y_d \triangleq \Sigma^{K}_{j=1}   \langle  y_d, e_j \rangle e_j$ and $\hat y_{0,T}\triangleq \Sigma^{K}_{j=1}  e^{\lambda_j T} \langle  y_d, e_j \rangle e_j$. Then, one can easily check that $\hat y_d$ and $y_{0,T}$ verify (\ref{lemma4-1-3}) for some $C \triangleq C(\Omega,b,y_d,r)>0$.

\textit{Step 2. To prove that for each $n \geq N_0$,
\begin{eqnarray}\label{lemma4-1-12}
  N(T, y_0, \hat Q)  \leq  N(T, y_{0}, Q_n)
\end{eqnarray} }
$\quad~\,$Since $Q_n\in \mathcal{F}$ for each $n \geq N_0$, we find
from Theorem \ref{Lemma-exist-op-NP}
that for each $n \geq N_0$, $(NP)^{T, y_{0}}_{Q_n}$ has at least one minimal norm control. Since $S_n\subset \bar S$ for each $n$, we have that $Q_n\subset \hat Q$ for all $n\in\mathbb N^+$. Thus, when $n\geq N_0$,
each minimal norm control to $(NP)^{T, y_{0}}_{Q_n}$ is an admissible control to  $(NP)^{T,y_{0}}_{\hat Q}$. This, along with the optimality of $N(T, y_0, \hat Q)$, leads to  (\ref{lemma4-1-12}).

\vskip5pt
\textit{Step 3. To prove that  for each $\lambda \in(0,1)$, there is $N_\lambda \geq N_0$ (independent of $T\in [a,b]$) so that when $n \geq N_\lambda$,
\begin{eqnarray}\label{lemma4-1-4}
  N(T, y_{0}, Q_n)  - N(T,  y_0, \hat Q)
 \leq  C(\Omega,\omega,a)  (1-\lambda)\| y_0- \hat y_{0,T}\|
\end{eqnarray}
for some $ C(\Omega,\omega,a)>0$ independent of $T\in [a,b]$ and $\lambda \in (0,1)$, where $\hat y_{0,T}$ is given by Step 1}

Let $u^*$ be a minimal norm control to $(NP)^{T, y_0}_{\hat Q}$. (Its existence
is ensured by (\ref{theorem4-q-1}) and Theorem \ref{Lemma-exist-op-NP}.)
Then
\begin{equation}\label{lemma4-1-5}
  \hat y(T;  y_0, u^*) \in \hat Q
  \;\;\mbox{and}\;\;
  \|u^*\|_{L^\infty(0,T;L^2(\Omega))} = N(T,  y_0, \hat Q).
\end{equation}
Let $\hat y_d$ be given by Step 1. Arbitrarily fix $\lambda \in(0,1)$.
At the end of the proof of this lemma, we will
prove the following conclusion: There is  $N_\lambda \geq N_0$ so that
\begin{eqnarray}\label{lemma4-1-85-1}
 \lambda (\hat Q-\{\hat y_d\}) \subset Q_n-\{\hat y_d\},\;\;\mbox{when}\;\;n \geq N_\lambda.
\end{eqnarray}
 We now suppose that (\ref{lemma4-1-85-1}) is true.
Then, arbitrarily fix $n\geq N_\lambda$. From the second conclusion in (\ref{lemma4-1-3}),
 the first conclusion in (\ref{lemma4-1-5}) and (\ref{lemma4-1-85-1}), we find that \begin{eqnarray}\label{lemma4-1-7}
\hat y(T; \lambda (y_0-\hat y_{0,T}), \lambda u^*)
=\lambda (\hat y(T;y_0,u^*)-\hat y_d)
\in  Q_n-\{\hat y_d\}.
 \end{eqnarray}
Meanwhile,  since $a\leq T\leq b$, by the $L^{\infty}$-null controllability for the heat equation (see \cite[Proposition 3.2]{FZ}), there is
$v_n \in {L^\infty(0,T;L^2(\Omega)})$, with supp\,$v_n\subset(0,a)$, so that
\begin{eqnarray}\label{lemma4-1-8-1}
\hat y(T; (1-\lambda)( y_0-\hat y_{0,T}), v_n )=0
\end{eqnarray}
and so that
\begin{eqnarray}\label{lemma4-1-8-2}
~~~~~\|v_n\|_{L^\infty(0,T;L^2(\Omega))}
\leq C(\Omega,\omega,a)(1-\lambda)\| y_0-\hat y_{0,T}\|\;\;\mbox{for some}\;C(\Omega, \omega, a)>0.
\end{eqnarray}
Now, it follows from the second conclusion in (\ref{lemma4-1-3}), (\ref{lemma4-1-8-1}) and (\ref{lemma4-1-7}) that
\begin{eqnarray*}\label{lemma4-1-9}
 \hat y(T;y_{0}, \lambda u^*+v_n)
&=& e^{\Delta T} \hat y_{0,T}
+ \hat y(T;y_{0}-\hat y_{0,T}, \lambda u^*+v_n)
\nonumber\\
&=& \hat y_d+\hat y(T;\lambda ( y_0-\hat y_{0,T}), \lambda u^*)
\in Q_n.
\end{eqnarray*}
Thus, $\lambda u^*+v_n$ is an admissible control to $(NP)^{T,y_{0}}_{Q_n}$,
which, along with  the optimality of $N(T, y_{0}, Q_n)$, yields that
\begin{eqnarray*}\label{lemma4-1-10}
  N(T, y_{0}, Q_n) \leq \|\lambda u^*+v_n\|_{L^\infty(0,T;L^2(\Omega))}.
\end{eqnarray*}
This, together with the second conclusion in (\ref{lemma4-1-5}) and  (\ref{lemma4-1-8-2}), yields that
\begin{eqnarray*}
   N(T, y_{0}, Q_n)
  \leq \lambda N(T,  y_0, \hat Q)
  + C(\Omega,\omega,a) (1-\lambda)\| y_0-\hat y_{0,T}\|.
\end{eqnarray*}
Since $\lambda\in(0,1)$, the above inequality leads to (\ref{lemma4-1-4}). This ends the proof of Step 3.

\textit{Step 4. To verify (\ref{theorem4-1-a})}

Given $\varepsilon\in (0,1)$, it follows  by (\ref{lemma4-1-4}) and (\ref{lemma4-1-12}) that when $n\geq N_{1-\varepsilon}$ (where $N_{1-\varepsilon}$ is given by Step 3, with $\lambda=1-\varepsilon$),
    \begin{eqnarray*}\label{lemma4-1-14}
 | N(T,  y_0, \hat Q) - N(T, y_{0}, Q_n)|
  \leq C(\Omega,\omega,a) \varepsilon \big( \| y_0\| +\|\hat y_{0,T}\| \big),
\end{eqnarray*}
where $C(\Omega,\omega,a)$ is independent of $T\in [a,b]$.
This, along with the third conclusion in (\ref{lemma4-1-3}), leads to (\ref{theorem4-1-a}).

In summary, we conclude that if we can show (\ref{lemma4-1-85-1}), then the proof of Lemma~\ref{theorem4-1} is completed.

The proof of   (\ref{lemma4-1-85-1}) is as follows:
Arbitrarily fix $\lambda\in (0,1)$. Write
\begin{eqnarray}\label{0831-yd-decom}
  \hat y_d=\hat y_{d,1}+\hat y_{d,2}
  \;\;\mbox{with}\;\;
  \hat y_{d,1}\in E
  \;\;\mbox{and}\;\;
  \hat y_{d,2}\in E^\perp.
\end{eqnarray}
We divide the proof of (\ref{lemma4-1-85-1}) by two parts.

  \textit{Part 1. To show that (\ref{lemma4-1-85-1}) holds if
  $\lambda(\bar S - \{\hat y_{d,1}\}) \subset S_n - \{\hat y_{d,1}\}$}

Assume that $\lambda(\bar S - \{\hat y_{d,1}\}) \subset S_n - \{\hat y_{d,1}\}$. Since  $\hat Q = \bar S \oplus B_{E^\perp}$ and $Q_n = S_n \oplus B_{E^\perp}$ (see (\ref{theorem4-q-1}) and (\ref{wang3.2})), it follows that
\begin{equation}\label{T-8-28-2}
  P_E \big( \lambda(\hat Q - \{\hat y_d\}) \big) \subset P_E \big( Q_n - \{\hat y_d\} \big),
\end{equation}
where $P_{E}$ denotes the orthogonal projection from $L^2(\Omega)$ onto $E$.
Next, we claim that
\begin{equation}\label{T-8-28-3}
  P_{E^\perp} \big( \lambda(\hat Q - \{\hat y_d)\} \big) \subset P_{E^\perp} \big( Q_n - \{\hat y_d\} \big),
\end{equation}
where $P_{E^\perp}$ denotes the orthogonal projection from $L^2(\Omega)$ onto $E^\perp$.
Indeed, since $\hat Q
 = \bar S \oplus B_{E^\perp}$  and $Q_n = S_n \oplus B_{E^\perp}$ (see (\ref{theorem4-q-1}) and (\ref{wang3.2})), we find that in order to show (\ref{T-8-28-3}), it suffices to prove that
\begin{equation}\label{T-8-28-4}
  \lambda(B_{E^\perp} - \{\hat y_{d,2}\}) \subset B_{E^\perp} - \{\hat y_{d,2}\}.
\end{equation}
To prove (\ref{T-8-28-4}), we use  (\ref{lemma4-1-3}), (\ref{0815-th3.2-yd}) and (\ref{wang3.2}) to  get that $B_{r/2}(\hat y_d)\subset S_{N_0} \oplus B_{E^\perp}$. This, along with the definition of $\hat y_{d,2}$ (see (\ref{0831-yd-decom})), yields that  $\hat y_{d,2}$ belongs to the interior of $B_{E^\perp}$. From this, one can  directly check that (\ref{T-8-28-4}) holds. So (\ref{T-8-28-3}) is true.
 Now the conclusion in Part 1 follows from (\ref{T-8-28-2}) and (\ref{T-8-28-3}) at once.

  \textit{Part 2. To show that there exists $N_\lambda\geq N_0$ so that for each $n \geq N_\lambda$,
  \begin{eqnarray}\label{0815-appendix-1}
 \lambda (\bar S -\{\hat y_{d,1}\}) \subset S_n-\{\hat y_{d,1}\}
\end{eqnarray}}
$\quad~\,$By contradiction, we suppose that (\ref{0815-appendix-1}) were not true. Then there would be two sequences $\{n_k\}_{k=1}^{+\infty}$ and $\{z_k\}_{k=1}^{+\infty} \subset \bar S-\{\hat y_{d,1}\}$ so that
\begin{eqnarray*}\label{0815-appendix-2}
 \lambda z_k \not\in S_{n_k}-\{\hat y_{d,1}\}
 \;\;\mbox{for each}\;\; k.
\end{eqnarray*}
For each $k$, by the Hahn-Banach separation theorem,   there exists $f_k\in E^* \setminus \{0\}$, with $\|f_k\|_{E^*}=1$, so that
\begin{eqnarray}\label{0815-appendix-3}
 \langle f_k, z \rangle_{E^*,E} < \langle f_k, \lambda z_k \rangle_{E^*,E},\;\;\forall\; z\in  S_{n_k}-\{\hat y_{d,1}\}.
\end{eqnarray}

Next, by (\ref{theorem4-q-1}) and the definition of $\mathcal F$,  we obtain that $\bar S$  is bounded in $E$ and so is the sequence $\{z_k\}_{k=1}^{+\infty}$. Since $E$ is of finite dimension, there exists a subsequence of $\{(z_k,f_k)\}_{k=1}^ {+\infty}$, still denoted in the same manner,  so that
\begin{eqnarray}\label{0815-appendix-4}
 \hat z = \lim_{k\rightarrow +\infty} z_{k}
 \in \bar S- \{\hat y_{d,1}\}
 \;\;\mbox{and}\;\;
 \hat f=\lim_{k\rightarrow +\infty} f_k
 \;\;\mbox{in}\;\;
 E^*
\end{eqnarray}
for some $(\hat z,\hat f)\in E\times (E^*\setminus\{0\})$.
Since $S=\cup_{n=1}^{+\infty} S_{n}$ and $\{S_n\}$ is increasing, we see that
 for each $z\in S- \{\hat y_{d,1}\}$,
  there is $k_z\in\mathbb N^+$ so that
  $z\in S_{n_k}- \{\hat y_{d,1}\}$ for all $k\geq k_z$. Thus, by  (\ref{0815-appendix-4}) and (\ref{0815-appendix-3}), we find that for each $z\in S- \{\hat y_{d,1}\}$,
\begin{eqnarray*}
 \langle \hat f, z \rangle_{E^*,E}=\lim_{k\rightarrow +\infty} \langle f_k, z \rangle_{E^*,E}
 \leq \lim_{k\rightarrow +\infty}\langle f_k, \lambda z_k \rangle_{E^*,E}
 =\langle \hat f, \lambda \hat z \rangle_{E^*,E}.
\end{eqnarray*}
This yields that
\begin{eqnarray}\label{0815-appendix-5}
 \langle \hat f, z \rangle_{E^*,E} \leq \langle \hat f, \lambda \hat z \rangle_{E^*,E},\;\;\forall\; z\in  \bar S-\{\hat y_{d,1}\}.
\end{eqnarray}
Since $\hat z\in \bar S-\{\hat y_{d,1}\}$ (see  (\ref{0815-appendix-4})), by taking $z=\hat z$ in (\ref{0815-appendix-5}), we see that
\begin{eqnarray*}
 (1-\lambda) \langle\hat f,  \hat z \rangle_{E^*,E} \leq 0,
 \;\;\mbox{i.e.,}\;\;
 \langle\hat f,  \hat z \rangle_{E^*,E} \leq 0.
\end{eqnarray*}
This, as well as (\ref{0815-appendix-5}), yields that
\begin{eqnarray}\label{0815-appendix-6}
  \langle \hat f, z \rangle_{E^*,E} \leq \langle \hat f, \lambda \hat z \rangle_{E^*,E} \leq 0,\;\;\forall\; z\in  \bar S-\{\hat y_{d,1}\}.
\end{eqnarray}
Meanwhile, by (\ref{lemma4-1-3}), (\ref{0815-th3.2-yd}) and (\ref{wang3.2}), we get that $B_{r/2}(\hat y_d)\in S_{N_0}\oplus B_{E^\perp}\subset \bar S \oplus B_{E^\perp}$. From this and the definition of $\hat y_{d,1}$ (see (\ref{0831-yd-decom})), we see that  $\hat y_{d,1}$ belongs to the interior of $\bar S$.
 By this and (\ref{0815-appendix-6}), we find that $\hat f=0$ in $E^*$, which leads to a contradiction. Therefore, (\ref{0815-appendix-1}) is true. This ends the proof of Part 2.

 Finally, by the conclusions in Part 1 and Part 2, we obtain   (\ref{lemma4-1-85-1}). This ends  the proof of Lemma~\ref{theorem4-1}.
\end{proof}

 To construct the desired $Q \in \mathcal F$ and $y_0 \in L^2(\Omega) \setminus Q$  in Theorem \ref{theorem4-3}, we need the following result.

\begin{lemma}\label{lemma4-2}
Let
$
\alpha \geq 2.
$
Let
\begin{eqnarray}\label{lemma4-2-2}
h(x) &\triangleq& x^\alpha,\; x>0,\\
  \nonumber
  G_h \triangleq \{(x_1,x_2)\in\mathbb R^2: x_1>0, x_2 \geq h(x_1) \}
  &,&
  \partial G_h \triangleq \{(x,h(x))\in\mathbb R^2: x>0\}.
\end{eqnarray}
 Then there exist two disjoint closed disks $D_{1}, D_{2} \subset  G_h$ so that they are respectively tangent to $\partial G_h$ at points $p_1$ and $p_2$ and so that $ conv\big( D_{1}\cup D_{2})\cap\partial G_h=\{p_1,p_2\}$, where $conv\big( D_{1}\cup D_{2})$ denotes the convex hull of $D_{1}\cup D_{2}$.

 \end{lemma}

 \begin{proof}
 Arbitrarily fix two different points $p_1$ and $p_2$ on  $\partial G_h$. Because the curve $\partial G_h$ is  smooth and the curvature of $\partial G_h$  at $p_i$, with $i=1,2$, is finite (which follow from (\ref{lemma4-2-2}) at once), we can find two disjoint closed disks $D_1$ and $D_2$ in $G_h$ so that $D_i$  is tangent to $\partial G_h$ at $p_i$, with $i=1,2$ (see, for instance, the contexts on Pages 354-355 in \cite{Gilbarg}).
 From this, we obtain that
 \begin{eqnarray}\label{WWang3.16}
 D_1 \cap D_2 = \emptyset, \;\; \partial G_h\cap D_1=\{p_1\}\;\;\mbox{and}\;\;\partial G_h\cap D_2=\{p_2\}.
 \end{eqnarray}

  We claim that  $ conv \big( D_{1}\cup D_{2})\cap\partial G_h=\{p_1,p_2\}$.
 Indeed, it is clear that
 \begin{eqnarray}\label{wang3.16}
 \{p_1,p_2\} \subseteq conv \big( D_{1}\cup D_{2})\cap\partial G_h.
 \end{eqnarray}
 Arbitrarily fix $p$ in $ conv \big( D_{1}\cup D_{2})\cap \partial G_h$.
 Since $D_{1}$ and $D_{2}$ are convex, by the definition of $conv \big( D_{1}\cup D_{2})$, one can easily check that
  \begin{eqnarray}\label{wang3.17}
 p=\mu q_1+(1-\mu)q_2\;\;\mbox{for some}\;\; \mu\in [0,1],\; q_1\in D_1,\; q_2\in D_2.
 \end{eqnarray}
 We now show  that
  \begin{eqnarray}\label{wang3.18}
 \{\lambda q_1+(1-\lambda)q_2\; :\; \lambda\in (0,1)\} \cap \partial G_h=\emptyset.
 \end{eqnarray}
 Indeed, in the first case that  $q_1\in \partial G_h$ and $q_2\in \partial G_h$,
 since $q_1\neq q_2$ (which follows from (\ref{WWang3.16})),
 by  the strict convexity of $h$ (which follows from  (\ref{lemma4-2-2})
 at once), we obtain
 (\ref{wang3.18});
 In the second  case that either  $q_1\not\in \partial G_h$ or $q_2\not\in \partial G_h$, we can assume, without loss of generality, that $q_1\not\in \partial G_h$. Write $q_1\triangleq (a_1,b_1)$ and $q_2\triangleq (a_2,b_2)$. Since $q_1\in G_h\setminus\partial G_h$ and $q_2\in G_h$, it follows from the definition $G_h$ and $\partial G_h$ (see (\ref{lemma4-2-2})) that $b_1>h(a_1)$ and $b_2\geq h(a_2)$. This, along with the convexity of $h$, indicates that for each $\lambda\in(0,1)$,
 \begin{eqnarray*}
  \lambda b_1+(1-\lambda)b_2> \lambda h(a_1)+(1-\lambda)h(a_2)\geq h(\lambda a_1+(1-\lambda)a_2).
 \end{eqnarray*}
 This implies that for each $\lambda\in(0,1)$, $\lambda q_1+(1-\lambda)q_2\not\in\partial G_h$. Hence, (\ref{wang3.18}) holds in the second case. In summary, we conclude that (\ref{wang3.18}) is true.

 Finally,  since $p\in \partial G_h$, it follows by (\ref{wang3.17}) and (\ref{wang3.18}) that $p$ is either  $q_1$ or $q_2$, from which, we see that $p$ is in either $\partial G_h\cap D_1$ or $\partial G_h\cap D_2$. This, along with
 (\ref{WWang3.16}), yields that $p$ is either  $p_1$ or $p_2$, which, together with
 (\ref{wang3.16}), leads to that $ conv \big( D_{1}\cup D_{2})\cap\partial G_h=\{p_1,p_2\}$. Thus,
 we end the proof of this lemma.
 \end{proof}

We are now on the position to prove Theorem \ref{theorem4-3}.
\begin{proof}
Let $(w_0, \hat Q) \in L^2(\Omega) \times \mathcal F$.
 We say that \textit{the function $t\rightarrow N(t, w_0, \hat Q)$ holds the property $\mathcal \backsim_{t_1,t_2}^{s_1,s_2}$, where $0<s_1<t_1<s_2<t_2<+\infty$}, if
 \begin{eqnarray}\label{W-property}
 0 = N(t_2, w_0, \hat Q) \leq N(t_1, w_0, \hat Q)
< N(s_2, w_0, \hat Q)
< \inf_{0 < t \leq s_1}N(t, w_0, \hat Q).
\end{eqnarray}
(In plain language, (\ref{W-property}) means that the function $t \rightarrow N(t,y_0,\hat Q)$ grows like a wave ``$\backsim$".) This property plays an important role in this proof. We prove Theorem \ref{theorem4-3} by two steps as follows:

\textit{Step 1. To show that there is $Q\in\mathcal F$, $y_0\in L^2(\Omega)\setminus Q$ and $(\tau_1,T_1,\tau_2,T_2)$ (with $0<\tau_1<T_1<\tau_2<T_2<\infty$) so that   the function $t \rightarrow N(t,y_0,Q)$ holds the property $\backsim_{T_1,T_2}^{\tau_1,\tau_2}$, and so that  $N(t,y_0,Q)>0$ for each $t\in(0,T_2)$}

Let $\{\lambda_j\}_{j=1}^{+\infty}$ and $\{e_j\}_{j=1}^{+\infty}$
be given by Step 1 of the proof of Lemma~\ref{theorem4-1} (see (\ref{lemma4-1-25})).
The proof of Step 1 is divided into the following four substeps.

\textit{Substep 1.1. To show that there exists  $(y_0, Q_1) \in L^2(\Omega) \times \mathcal F$ and $(T_1, T_2) \in \mathbb R^+ \times \mathbb R^+$, with $ T_1 < T_2 $, so that
\begin{equation}\label{theorem4-6-21}
  y_0 \in L^2(\Omega) \setminus Q_1
\end{equation}
and so that
\begin{equation}\label{theorem4-6-21-1}
   \{e^{\Delta t}y_0~:~t \in [0,+\infty)\} \cap Q_1
 = \{e^{\Delta T_1}y_0,e^{\Delta T_2}y_0\}
\end{equation}}
$\quad~\,$First, since
$\lim_{n \rightarrow +\infty} \lambda_n = +\infty$, we can fix $k \in \mathbb N^+$ so that
\begin{equation}\label{theorem4-3-2}
  \alpha \triangleq \lambda_{k} / \lambda_1 \geq 2.
\end{equation}
Let $h, G_h$ and $\partial G_h$ be defined in Lemma \ref{lemma4-2}, where $\alpha$ is given by (\ref{theorem4-3-2}), i.e.,
\begin{eqnarray}\label{T-8-30-1}
 h(x) &\triangleq& x^\alpha,\; x>0,\\
  \nonumber
G_h \triangleq \{(x_1,x_2)\in\mathbb R^2: x_1>0, x_2 \geq h(x_1) \}
  &,&
  \partial G_h \triangleq \{(x,h(x))\in\mathbb R^2: x>0\}.
\end{eqnarray}
 Then, according to Lemma \ref{lemma4-2},  there exist two disjoint closed disks $D_1, D_2 \subset G_h$ so that $D_1$ and $D_2$ are respectively tangent to $\partial G_h$ at $p_1 \triangleq (a_1, a_1^\alpha) \in \mathbb R^2$ and $p_2 \triangleq (a_2, a_2^\alpha) \in \mathbb R^2$, with $a_1 > a_2$ and so that
\begin{equation}\label{T-8-25-1}
 conv(D_1 \cup D_2) \cap \partial G_h = \{p_1, p_2\}.
\end{equation}
Furthermore, we see from \cite[Theorem 1.1.10 on Page 6]{Schneider}  that
\begin{eqnarray}\label{WanG3.46}
\overline{conv (D_1\cup D_2)}=conv (D_1\cup D_2).
\end{eqnarray}

Let $E \triangleq span
\{e_1, e_k\}$. Write $E^\perp$ for its orthogonal subspace in $L^2(\Omega)$.
Define an isomorphism $\mathcal I_{E}: \mathbb R^2 \rightarrow E$ by
\begin{eqnarray}\label{T-8-22-1}
  \mathcal I_{E}(a,b)=a e_1 + be_k\;\;\mbox{for each}\;\; (a,b)\in \mathbb R^2.
\end{eqnarray}
Choose  $a_0 > a_1$ large enough so that
\begin{equation}\label{T-8-22-3}
 p_0 \triangleq (a_0, a_0^\alpha) \in \partial G_h \setminus conv(D_1 \cup D_2).
\end{equation}
 We define  $(y_0, Q_1) \in L^2(\Omega) \times \mathcal F$ and $(T_1, T_2) \in \mathbb R^+ \times \mathbb R^+$ in the following manner:
\begin{equation}\label{T-8-22-5}
  y_0 \triangleq \mathcal I_{E}(p_0),\;\;
  Q_1 \triangleq S \oplus B_{E^\perp},\;\;
  T_1 \triangleq  \frac{1}{\lambda_1}\ln(\frac{a_0}{a_1})
   \;\;\mbox{and}\;\;
  T_2 \triangleq \frac{1}{\lambda_1}\ln(\frac{a_0}{a_2}),
\end{equation}
where $S \triangleq \mathcal I_{E} \big(conv( D_1 \cup D_2)\big)$ and $B_{E^\perp}$ denotes the closed unit ball in $E^\perp$.

Now, we claim that the above-mentioned $(y_0, Q_1)$ and $(T_1, T_2)$ satisfy
 (\ref{theorem4-6-21}) and (\ref{theorem4-6-21-1}).
 To prove  (\ref{theorem4-6-21}), we observe from the first equality in (\ref{T-8-22-5}), (\ref{T-8-22-3}) and the definition of $S$ that
 \begin{eqnarray*}
 y_0=\mathcal I_{E}(p_0)\notin \mathcal I_{E}(conv(D_1\cap D_2))=S\;\;\mbox{and}\;\; y_0\in E.
 \end{eqnarray*}
 These, along with the definition of $Q_1$ (see (\ref{T-8-22-5})), lead to (\ref{theorem4-6-21}).

 To show (\ref{theorem4-6-21-1}), we use the definitions of $y_0,\, \mathcal I_{E},\, p_0,\, \alpha$ and $\partial G_h$ (see (\ref{T-8-22-5}),(\ref{T-8-22-1}),(\ref{T-8-22-3}),(\ref{theorem4-3-2}) and (\ref{T-8-30-1}), respectively) to find that
\begin{eqnarray}\label{T-8-22-7}
  \{e^{\Delta t}y_0~:~t \in [0,+\infty)\} &=& \{e^{-\lambda_1 t}a_0e_1 + e^{-\lambda_k t}a_0^\alpha e_k ~:~t \in [0,+\infty)\}\\
\nonumber
                                          &=& \{(e^{-\lambda_1 t}a_0)e_1 +
  (e^{-\lambda_1 t}a_0)^\alpha e_k ~:~t \in [0,+\infty)\}\\
\nonumber
                                          &\subset& \mathcal I_{E}(\partial G_h).
\end{eqnarray}
Meanwhile, by (\ref{T-8-22-5}), (\ref{T-8-25-1}), (\ref{T-8-22-1}) and (\ref{theorem4-3-2}), we can directly check that
\begin{eqnarray}\label{T-8-31-8}
   Q_1 \cap \mathcal I_{E}(\partial G_h)
   &=& \{\mathcal I_{E}(p_1), \mathcal I_{E}(p_2)\}
                                  = \{a_1e_1+a_1^{\alpha}e_k, a_2e_1+a_2^{\alpha}e_k\}
   \nonumber\\
                                 &=& \{(e^{-\lambda_1 T_1}a_0)e_1 +
  (e^{-\lambda_1 T_1}a_0)^\alpha e_k, (e^{-\lambda_1 T_2}a_0)e_1 +
  (e^{-\lambda_1 T_2}a_0)^\alpha e_k\}
  \nonumber\\
                                 &=& \{e^{\Delta T_1}y_0, e^{\Delta T_2}y_0\}.
\end{eqnarray}
This, along with (\ref{T-8-22-7}),
yields that $ \{e^{\Delta t}y_0~:~t \in [0,+\infty)\}\cap Q_1  \subset \{e^{\Delta T_1}y_0, e^{\Delta T_2}y_0\}$. The reverse is clear. Hence, (\ref{theorem4-6-21-1}) is true.

\textit{Substep 1.2. To show that the function $t\rightarrow N(t,y_0,Q_1)$ holds the property $\backsim_{T_1,T_2}^{\tau_1,\tau_2}$ for some  $(\tau_1,\tau_2) \in \mathbb R^+ \times \mathbb R^+$, with $\tau_1<T_1<\tau_2<T_2$, and satisfies that
\begin{eqnarray}\label{theorem4-3-16-2}
N(T_1,y_0,Q_1) = N(T_2,y_0,Q_1) = 0
\end{eqnarray}}
$\quad~\,$First, we verify (\ref{theorem4-3-16-2}).
 From (\ref{theorem4-6-21-1}), we see that
$\hat y(T_1;y_0,0)\in Q_1$ and $\hat y(T_2;y_0,0)\in Q_1$.
These, along with (\ref{I-1}), lead to  (\ref{theorem4-3-16-2}).

We next show the existence of the desired pair $(\tau_1,\tau_2)\in \mathbb R^+ \times \mathbb R^+$. Arbitrarily fix $\tau_2\in(T_1,T_2)$. It is clear that $N(\tau_2,y_0,Q_1)<\infty$ (see, for instance, Theorem~\ref{Lemma-exist-op-NP}). Since
$y_0 \in L^2(\Omega) \setminus Q_1$ (see (\ref{theorem4-6-21})), it follows by Theorem \ref{Theorem-Lip-NP} that $\lim_{t\rightarrow 0^+} N(t,y_0,Q_1 ) = +\infty$. Thus, there is $\tau_1 \in(0,T_1)$ so that
 \begin{equation}\label{T-8-31-7}
  N(\tau_2,y_0,Q_1)< \inf_{0 < t \leq \tau_1} N(t, y_0, Q_1).
\end{equation}
Meanwhile, it follows from (\ref{theorem4-6-21-1}) that $\hat y (\tau_2, y_0, 0)=e^{\Delta \tau_2}y_0 \notin Q_1$, we see from (\ref{I-1}) that $N(\tau_2,y_0,Q_1) > 0$. This, together with (\ref{T-8-31-7}), (\ref{theorem4-3-16-2}) and
the definition of the property $\backsim_{T_1,T_2}^{\tau_1,\tau_2}$ (see (\ref{W-property})), indicates that the function $ t \rightarrow N(t,y_0,Q_1)$ holds the property $\backsim_{T_1,T_2}^{\tau_1,\tau_2}$. Thus, we end the proof of Substep 1.2.

\textit{Substep 1.3. To show the existence of  $Q_2 \in \mathcal F$, with $Q_2 \subset Q_1$, so that $\{e^{\Delta t}y_0~:~t \in [0,+\infty)\} \cap Q_2 = \{e^{\Delta T_2}y_0\}$ and  the function $t \rightarrow N(t,y_0,Q_2)$ holds the property $\backsim_{T_1,T_2}^{\tau_1,\tau_2}$}

First of all, by Substep 1.2, we have that
 \begin{eqnarray}\label{160904-100}
  N(T_1, y_0, Q_1) < N(\tau_2,y_0,Q_1)< \inf_{0 < t \leq \tau_1} N(t, y_0, Q_1).
 \end{eqnarray}
$\quad~\,$
We will use some perturbation of $Q_1$ to construct $Q_2$. Let
$D_1\triangleq \{z \in \mathbb R^2:\; \|z-\hat z\|_{\mathbb R^2} \leq \hat r\}$
 (for some $\hat z \in \mathbb R^2$ and $\hat r > 0$) be the closed disk  given in the proof Substep 1.1.
For each $\alpha \in [0,1)$, we define
\begin{eqnarray}\label{T-8-30-3}
  &D_{\alpha}& \triangleq \{z \in \mathbb R^2:\; \|z-\hat z\|_{\mathbb R^2} \leq r_{\alpha} \triangleq \alpha \hat r\} \subset D_1,\\
  \nonumber
  &S_{\alpha}& \triangleq  \mathcal I_{E} \big(conv(D_{\alpha} \cup D_2 \big)
  \;\;\mbox{and}\;\;
  Q_{\alpha} \triangleq S_{\alpha} \oplus B_{E^\perp},
\end{eqnarray}
where $B_{E^\perp}$ denotes the closed unit ball in $E^\perp$. It follows from \cite[Theorem 1.1.10 on Page 6]{Schneider}
that for each $\alpha \in [0,1)$, $D_\alpha \cup D_2$ is closed in $\mathbb{R}^2$. This, along with  the definition of $\mathcal I_E$ (see (\ref{T-8-22-1})), yields that for each $\alpha \in [0,1)$, $S_\alpha$ is closed in $E$. Thus, by the definition of $\mathcal F$, one can check that $Q_\alpha \in \mathcal F$ for each $\alpha \in (0,1)$.

 Arbitrarily fix $\alpha \in [0,1)$.
 Since $Q_1 \triangleq S \oplus B_{E^\perp}$, where $S=\mathcal I_E \big(conv( D_1 \cup D_2)\big)$ (see (\ref{T-8-22-5})), we see from (\ref{T-8-30-3}) that
\begin{equation}\label{T-8-31-10}
  Q_{\alpha} \subset Q_1.
\end{equation}
 Because $D_1$ and $D_2$ are respectively tangent to $\partial G_h$ at $p_1 \in \mathbb R^2$ and $p_2 \in \mathbb R^2$ (see the proof of Substep 1.1), we deduce from (\ref{T-8-30-3}) that
\begin{equation}\label{T-8-31-11}
  p_1 \in D_1 \setminus D_{\alpha} \;\;\mbox{and}\;\; p_2 \in D_2.
\end{equation}
 By (\ref{T-8-31-8}),  we have that $\mathcal{I}_E(p_i)=e^{\Delta T_i}y_0$, $i=1,2$. This, along with
  the definition of $S_{\alpha}$ (see (\ref{T-8-30-3})) and (\ref{T-8-31-11}), yields that
\begin{eqnarray*}
  e^{\Delta T_1}y_0 =  \mathcal I_{E}(p_1) \notin S_{\alpha}
  \;\;\mbox{and}\;\;
  e^{\Delta T_2}y_0 = \mathcal I_{E}(p_2) \in S_{\alpha}.
\end{eqnarray*}
Then,  from the definition $Q_{\alpha}$ (see (\ref{T-8-30-3})), we find that  $e^{\Delta T_1}y_0 \notin Q_{\alpha}$ and $e^{\Delta T_2}y_0 \in Q_{\alpha}$.
These, along with (\ref{theorem4-6-21-1}) and (\ref{T-8-31-10}), yield that
\begin{equation}\label{T-8-2}
  \{e^{\Delta t}y_0~:~t \in [0,+\infty)\} \cap Q_{\alpha} = \{e^{\Delta T_2}y_0\}.
\end{equation}
Hence, we obtain  that  $\hat y(T_2, y_0, 0) \in Q_{\alpha}$ and $\hat y(T_1, y_0, 0) \notin Q_{\alpha} $.
From these and  (\ref{I-1}), we conclude   that
\begin{equation}\label{T-8-31-4}
   N(T_2, y_0, Q_{\alpha}) =0 < N(T_1, y_0, Q_{\alpha})\;\;\mbox{for each}\;\;\alpha\in [0,1).
\end{equation}

Next, we will use Lemma \ref{theorem4-1} to show that for each $t>0$, the map $\alpha\rightarrow N(t,y_0,Q_\alpha)$ is continuous at $\alpha=1$. For this purpose, we need to check that
\begin{eqnarray}\label{160904-1}
 \overline{\cup_{0\leq\alpha<1} S_{\alpha}}=S.
\end{eqnarray}
Since $\mathcal I_E$ is isometric from $\mathbb R^2$ onto $E$, from the definitions of $\{S_\alpha\}$ and $S$ (see (\ref{T-8-30-3}) and (\ref{T-8-22-5})), we see that to show (\ref{160904-1}), it suffices to prove that
\begin{eqnarray}\label{160904-2}
 \overline{\cup_{0\leq\alpha<1}  conv (D_{\alpha}\cup D_2)}=conv(D_1\cup D_2).
\end{eqnarray}
To show (\ref{160904-2}), we claim that
\begin{eqnarray}\label{L-9-6-1}
 \overline{ \cup_{0\leq\alpha<1}  conv (D_{\alpha}\cup D_2) } \subset
    conv (D_1\cup D_2);
\end{eqnarray}
and
\begin{eqnarray}\label{L-9-6-2}
 conv (D_1\cup D_2) \subset \overline{ \cup_{0\leq\alpha<1}  conv (D_{\alpha}\cup D_2) }.
\end{eqnarray}
To prove (\ref{L-9-6-1}), we first claim  that
\begin{eqnarray}\label{WanG3.43}
 \cup_{0\leq\alpha<1}  conv (D_{\alpha}\cup D_2) = conv (\cup_{0\leq\alpha<1} (D_{\alpha}\cup D_2)).
\end{eqnarray}
Indeed, on one hand, it is clear that
\begin{eqnarray}\label{160904-3}
  \cup_{0\leq\alpha<1}  conv (D_{\alpha}\cup D_2) \subset conv (\cup_{0\leq\alpha<1} (D_{\alpha}\cup D_2))\triangleq \mathcal A.
\end{eqnarray}
On the other hand,  for each $p \in \mathcal A$, there exists $N\in\mathbb N^+$, $\{p_i\}_{i=1}^N\subset \cup_{0\leq\alpha<1} (D_{\alpha}\cup D_2)$ and $\{\lambda_i\}_{i=1}^N \subset \mathbb R^+$ so that $p=\sum_{i=1}^N \lambda_i  p_i$ and $\sum_{i=1}^N \lambda_i=1$. Since $ p_i\in\cup_{0\leq\alpha<1} (D_{\alpha}\cup D_2)$ for each $i$, there exists $\alpha_i\in[0,1)$ so that $ p_i\in D_{\alpha_i}\cup D_2$. Let $\bar \alpha\triangleq \max\{\alpha_1,\cdots,\alpha_N\}$. Because $D_{\alpha_i}\subset D_{\bar\alpha}$ (see (\ref{T-8-30-3})), we find that $ p_i\in D_{\bar\alpha}$ for each $i$.
Thus, we get that $p=\sum_{i=1}^N \lambda_i  p_i \in conv (D_{\bar\alpha}\cup D_2))$. Hence, $\mathcal A \subset \cup_{0\leq\alpha<1}  conv (D_{\alpha}\cup D_2)$. This, along with  (\ref{160904-3}), leads to (\ref{WanG3.43}).

By (\ref{WanG3.43}) and the definitions of  $\{D_\alpha\}_{0\leq\alpha<1}$ and $D_1$ (see (\ref{T-8-30-3})), one can directly check that
\begin{eqnarray}\label{160904-4}
 \cup_{0\leq\alpha<1}  conv (D_{\alpha}\cup D_2) = conv ((\mbox{Int}\,D_{1})\cup D_2),
\end{eqnarray}
where $\mbox{Int}\,D_{1}$ is the interior of $D_1$.
From (\ref{160904-4}) and (\ref{WanG3.46}),
one can easily get (\ref{L-9-6-1}).

To show (\ref{L-9-6-2}), we arbitrarily fix $\hat p\in conv (D_1\cup D_2)$ and $\varepsilon>0$. Then  there exists $\hat N\in\mathbb N^+$, $\{\hat p_i\}_{i=1}^{\hat N}\subset D_1\cup D_2$ and $\{\hat \lambda_i\}_{i=1}^{\hat N} \subset \mathbb R^+$ so that $\hat p=\sum_{i=1}^{\hat N} \hat \lambda_i \hat p_i$ and $\sum_{i=1}^{\hat N} \hat \lambda_i=1$. Since $\hat p_i\in D_1\cup D_2$ for each $i$, we can find $\hat q_i \in (\mbox{Int}\,D_{1})\cup D_2$ so that $\|\hat p_i-\hat q_i\|_{\mathbb R^2} \leq \varepsilon$. Thus, we find that
\begin{eqnarray*}
 \|\hat p-  \sum_{i=1}^{\hat N} \hat \lambda_i \hat q_i\|_{\mathbb R^2} = \|\sum_{i=1}^{\hat N} \hat \lambda_i (\hat p_i-\hat q_i)\|_{\mathbb R^2}\leq \varepsilon.
\end{eqnarray*}
Since $\sum_{i=1}^{\hat N} \hat \lambda_i \hat q_i\in conv ((\mbox{Int}\,D_{1})\cup D_2)$ and $\varepsilon>0$ was arbitrarily fixed, the above yields that $\hat p\in \overline{ conv ((\mbox{Int}\,D_{1})\cup D_2) }$. Therefore, we have that
\begin{eqnarray*}
 conv (D_1\cup D_2)\subset \overline{ conv ((\mbox{Int}\,D_{1})\cup D_2) }.
\end{eqnarray*}
 This, along with (\ref{160904-4}), leads to (\ref{L-9-6-2}). Now,  (\ref{160904-2}) follows from (\ref{L-9-6-1}) and (\ref{L-9-6-2}) at once. Hence, (\ref{160904-1}) is true.

By (\ref{160904-1}), we can apply  Lemma \ref{theorem4-1} to  see that
 for each $t>0$, the map $\alpha\rightarrow N(t,y_0,Q_\alpha)$ is continuous at $\alpha=1$.
 From this  and (\ref{160904-100}), we find that there exists some $\alpha_0 \in [0, 1)$ so that
 \begin{equation}\label{T-8-31-5}
   N(T_1, y_0, Q_{\alpha_0}) < N(\tau_2,y_0,Q_{\alpha_0})< \inf_{0 < t \leq \tau_1} N(t, y_0, Q_1).
 \end{equation}
 We now deal with the last term on the right hand side of (\ref{T-8-31-5}).
 Arbitrarily fix $t\in (0,\tau_1]$.
 Since $Q_{\alpha_0} \subset Q_1$ (see (\ref{T-8-31-10})), we see that each admissible control to
$(NP)^{t,y_0}_{Q_{\alpha_0}}$ is also an admissible control for
$(NP)^{t,y_0}_{Q_1}$. Thus, $N(t, y_0, Q_1) \leq N(t, y_0, Q_{\alpha_0})$.
Hence, we have that
 \begin{eqnarray*}
   \inf_{0 < t \leq \tau_1} N(t, y_0, Q_1) &\leq& \inf_{0 < t \leq \tau_1} N(t, y_0, Q_{\alpha_0}).
 \end{eqnarray*}
 This, along with (\ref{T-8-31-5}) and (\ref{T-8-31-4}), yields that
 \begin{eqnarray*}
   0 = N(T_2, y_0, Q_{\alpha_0}) < N(T_1, y_0, Q_{\alpha_0}) < N(\tau_2,y_0,Q_{\alpha_0})< \inf_{0 < t \leq \tau_1} N(t, y_0, Q_{\alpha_0}),
 \end{eqnarray*}
from which and (\ref{W-property}), we see that the function $t \rightarrow N(t,y_0,Q_{\alpha_0})$ holds the property $\backsim_{T_1,T_2}^{\tau_1,\tau_2}$. This, along with (\ref{T-8-2}), leads to the conclusions in Substep 1.3, with $Q_2 = Q_{\alpha_0}$.

\vskip 5pt
\textit{Substep 1.4. To show the conclusions in Step 1
}

We will prove that $(y_0,Q_2)$ and $(\tau_1,T_1,\tau_2,T_2)$  satisfy the conclusions in Step 1.
 In Substep 1.3,
  we already proved that the function $t \rightarrow N(t,y_0,Q_2)$ holds the property $\backsim_{T_1,T_2}^{\tau_1,\tau_2}$. Thus,
  we only need to show that for each $t\in(0,T_2)$, $N(t,y_0,Q_2)>0$ . For this purpose, we arbitrarily fix $\hat t\in(0,T_2)$. By Substep 1.3, we have that $e^{\Delta \hat t}y_0\not\in Q_2$. This yields that $\hat y(\hat t;y_0,0)\not\in Q_2$, which, along with (\ref{I-1}), indicates that $N(\hat t,y_0,Q_2)>0$. This ends the proof of Step 1.

\textit{Step 2. To prove that the pair $(y_0,Q)$ in Step 1 verifies the conclusions (i) and (ii) in Theorem \ref{theorem4-3}}

Let  $(y_0,Q)$ and $(\tau_1,T_1,\tau_2,T_2)$  be given by Step 1. Since the function $t\rightarrow N(t,y_0,Q)$ holds the property $\backsim_{T_1,T_2}^{\tau_1,\tau_2}$, we can use (\ref{W-property}) to see that
\begin{eqnarray}\label{0902-property}
~~~~~~
\left\{\begin{array}{l}
 0 = N(T_2, y_0,  Q) \leq N(T_1, y_0,  Q)
< N(\tau_2, y_0,  Q)
< \inf_{0 < t \leq \tau_1}N(t, y_0,  Q),\\
0<\tau_1<T_1<\tau_2<T_2.
\end{array}
\right.
\end{eqnarray}
The conclusion (i) in Theorem \ref{theorem4-3} follows from (\ref{0902-property}) at once.

To show that $(y_0,Q)$ satisfies  the conclusion (ii) in Theorem \ref{theorem4-3}, we
define
\begin{eqnarray}\label{th3.1-s4-10}
 M_0 \triangleq   \inf_{\tau_1\leq t\leq \tau_2}N(t,y_0,Q).
\end{eqnarray}
We claim that there exists $\widehat T\in (\tau_1,\tau_2)$ so that
\begin{eqnarray}\label{th3.1-s4-2}
 \inf_{0< t\leq \tau_2} N(t,y_0,Q)=M_0=N(\hat T , y_0, Q)>0.
\end{eqnarray}
In fact, by (\ref{0902-property}), we have that
\begin{eqnarray*}
 \inf_{0< t\leq \tau_1}N(t,y_0,Q)>N(\tau_2,y_0,Q).
\end{eqnarray*}
From this and (\ref{th3.1-s4-10}), we get the first equality in (\ref{th3.1-s4-2}). To prove the second equality in (\ref{th3.1-s4-2}), we use Theorem \ref{Theorem-Lip-NP} to obtain that the function $t\rightarrow N(t,y_0,Q)$ is continuous over $[\tau_1,\tau_2]$. Thus, there exists $\widehat T\in [\tau_1,\tau_2]$ so that
\begin{eqnarray}\label{L-9-5-1}
 M_0=N(\hat T , y_0, Q).
\end{eqnarray}
Since $T_1\in(\tau_1,\tau_2)$ (see (\ref{0902-property})), we find from the above and the definition of $M_0$ (see  (\ref{th3.1-s4-10})) that $M_0\leq N(T_1,y_0,Q)$. By this and (\ref{0902-property}), we get that
\begin{eqnarray*}
  M_0 <\min\{N(\tau_1,y_0,Q), N(\tau_2,y_0,Q)\}.
\end{eqnarray*}
Thus, $\hat T \notin \{\tau_1, \tau_2\}$. This, along with (\ref{L-9-5-1}), yields the second equality in (\ref{th3.1-s4-2}). Finally, since $\widehat T<\tau_2$, it follows from (\ref{0902-property}) that $\widehat T<T_2$. From this and Step 1, we see that $N(\hat T , y_0, Q)>0$. In summary, we conclude that (\ref{th3.1-s4-2}) is true.

Next,  we see from (\ref{0902-property}) that $N(T_2, y_0, Q) = 0$. By this and the definition of $(\mathcal KN)_{y_0,Q}$ (see (\ref{def-KN})), we get that $(0,T_2)\in (\mathcal KN)_{y_0,Q}$, from which, it follows  that $(\mathcal KN)_{y_0,Q} \neq \emptyset$.

Finally, we show that $(\mathcal GT)_{y_0,Q}$ is not connected. For this purpose, we claim that
\begin{eqnarray}
 & & \inf \left\{T(M,y_0,Q)~:~M \geq M_0\right\}
 \leq \hat T<\tau_2; \label{th3.1-s4-4}
\\
 & & \tau_2
 \leq \inf \left\{T(M,y_0,Q)~:~0 \leq M < M_0 \right\}\leq T_2.
 \label{th3.1-s4-5}
\end{eqnarray}
To this end, by  (\ref{th3.1-s4-2}), we have that $N(\widehat T,y_0,Q)=M_0$. This, along with Theorem \ref{Theorem-dis-TP-NP}, indicates that for each $M\geq M_0$,
\begin{eqnarray*}
  T(M,y_0,Q) = \inf\{t\in\mathbb R^+~:~ N(t,y_0,Q)\leq M\} \leq \hat T.
\end{eqnarray*}
Since $\widehat T\in(\tau_1,\tau_2)$ (see (\ref{th3.1-s4-2})), the above leads to  (\ref{th3.1-s4-4}).
To show (\ref{th3.1-s4-5}), we find from (\ref{th3.1-s4-2}) that
\begin{eqnarray*}\label{th3.1-s4-11}
 0 < M_0 \leq N(t,y_0,Q)
 \;\;\mbox{for each}\;\; t\in(0,\tau_2].
\end{eqnarray*}
Then we see that for each $M\in[0,M_0)$,
\begin{eqnarray*}
 (0,\tau_2] \cap  \{t\in\mathbb R^+~:~ N(t,y_0,Q)\leq M\}
 =\emptyset.
\end{eqnarray*}
This, together with Theorem \ref{Theorem-dis-TP-NP}, yields that for each $M\in[0,M_0)$,
\begin{eqnarray}\label{th3.1-s4-12}
  T(M,y_0,Q) = \inf\{t\in\mathbb R^+~:~ N(t,y_0,Q)\leq M\} \geq \tau_2.
\end{eqnarray}
Meanwhile, by $N(T_2, y_0, Q) = 0 \leq M$ (see (\ref{0902-property})) and Theorem \ref{Theorem-dis-TP-NP}, we get that $T(M,y_0,Q) = \inf \mathcal J_M \leq T_2$. By this and (\ref{th3.1-s4-12}), we are led to (\ref{th3.1-s4-5}).

Now, we use (\ref{th3.1-s4-4}) and (\ref{th3.1-s4-5}) to prove that $(\mathcal GT)_{y_0,Q}$ is not connected. Let $\varepsilon\triangleq \frac{\tau_2-\hat T}{3}$.
 Define two sets in the following manner:
\begin{eqnarray*}
 \mathcal O_1\triangleq \big([0, +\infty)\times (0, \hat T+\varepsilon) \big) \cap (\mathcal GT)_{y_0,Q};
 \;\;
  \mathcal O_2 \triangleq \big( [0, +\infty) \times (\tau_2-\varepsilon,+\infty) \big) \cap (\mathcal GT)_{y_0,Q}.
\end{eqnarray*}
From these,  (\ref{th3.1-s4-4}), (\ref{th3.1-s4-5}) and the definition of $(\mathcal GT)_{y_0,Q}$ (see (\ref{def-GT})), one can easily check that
\begin{eqnarray}\label{WanG3.56}
 (\mathcal GT)_{y_0,Q}
 = \mathcal O_1\cup \mathcal O_2
 \;\;\mbox{and}\;\;
 \mathcal O_1\neq\emptyset,\,\mathcal O_2\neq\emptyset.
\end{eqnarray}
On the other hand, it is clear that $\mathcal O_1\cap \mathcal O_2=\emptyset$.
This, together with (\ref{WanG3.56}), indicates that $(\mathcal GT)_{y_0,Q}$ is not connected.

In summary, we conclude that the pair $(y_0,Q)$ satisfies the conclusion (i) and (ii)
in Theorem \ref{theorem4-3}. This completes the proof of Theorem \ref{theorem4-3}.
\end{proof}

\end{document}